\documentclass[a4paper,11pt, reqno]{amsart}
\usepackage{amsmath, amsthm, amssymb, amsfonts}
\usepackage[foot]{amsaddr}
\usepackage{eufrak}
\usepackage{euscript}
\usepackage{graphicx}
\usepackage[english]{babel}
\usepackage[utf8]{inputenc}
\usepackage[T1]{fontenc}
\usepackage{enumerate}
\usepackage{color}
\newtheorem{theorem}{Theorem}[subsection]
\newtheorem{proposition}{Proposition}
\newtheorem{definition}{Definition}
\newtheorem{lemma}{Lemma}[subsection]
\newtheorem{note}{Note}[section]
\newtheorem{remark}{Remarks}[section]
\newtheorem{corollary}{Corollary}[section]

\usepackage{setspace}

\usepackage[a4paper,top=2cm,bottom=1.5cm,left=1.5cm,right=1.5cm,marginparwidth=1cm]{geometry}

\def \H{\mathcal{H}}

\def \T{\mathbb{HT}}
\def \P{\mathcal{P}_{hg}}

\def \l {\lambda}
\def \R{\mathcal{TC}}
\def \PH{\mathcal{P(\H)}}

\title{On the spectral radius of some linear hypergraphs}
	\email[ Banerjee]{\textit {{\scriptsize anirban.banerjee@iiserkol.ac.in}}}

\author[Banerjee]{Anirban Banerjee } 
\address[Banerjee]{Department of Mathematics and Statistics, Indian Institute of Science Education and Research Kolkata, Mohanpur-741246, India}
\email[Sarkar ]{\textit {{\scriptsize  amiteshsarkarp@gmail.com}}}

\author[Sarkar]{Amitesh Sarkar} 
\address[Sarkar]{Department of Mathematics and Statistics, Indian Institute of Science Education and Research Kolkata, Mohanpur-741246, India}
\date{\today}

\date{\today}

\begin{document}
	\maketitle
	
\begin{abstract}
Here we study the spectral radii of some linear hypergraphs, that is, the maximum moduli of the eigenvalues of their corresponding adjacency matrices. We determine the hypertrees having the largest to seventh-largest spectral radii. The hypertrees with the largest and the second-largest spectral radii among all those with a given diameter are identified here. Unicyclic hypergraphs with a fixed cycle length having the largest, the second-largest, and the third-largest spectral radii are also determined. Furthermore, we also find which bicyclic and tricyclic linear hypergraphs have the largest and the second-largest spectral radii.

\end{abstract}
\section{Introduction}
A hypergraph $\H(V, E)$ consists of a vertex set $V$ and an edge set $E$, where elements of $E$ are non-empty subsets of the vertex set $V$. The hypergraph is considered ($m$-)uniform if any two elements in $E$ have the same cardinality($m$). 
Two vertices $i,j$ are said to be adjacent if they belong to an edge, and we denote it as $i\sim j$. A vertex $v\in V$ is incident to an edge $e\in E$ if $v\in e$.
We say two edges are adjacent if they intersect, i.e., if they have at least one vertex in common. A vertex is called a pendant vertex if the degree, i.e., the number of its adjacent vertices, is one, and a non-pendant vertex if its degree is more than one. 
A hypergraph $\H(V,E)$ is said to be linear if $|e_1\cap e_2|\leq 1,$ for any two distinct elements, $e_1,e_2\in {E}$. In a linear hypergraph, an edge is a pendant edge if it possesses exactly one non-pendant vertex. A non-pendant edge contains more than one non-pendant vertices.
An $m$-uniform loose path \cite{Peng2016} of length $l$ is an $m$-uniform hypergraph with the vertex set $V=\{1,2,\dots,l(m-1)+1\},$ and the edge set $E=\{\{i(m-1)+1,i(m-1)+2,\dots,i(m-1)+m\}:i=0,1,\dots,l-1\}.$ We represent it by $P_L(l)=v_1e_1v_2e_2\dots v_le_lv_{l+1}$, where the vertices $v_2,v_3,\dots,v_{l-1}$ are considered as the core vertices, and the remaining vertices are as the loose vertices of $P_L(l)$.
Similarly, an $m$-uniform loose cycle of length $l$ is an $m$-uniform hypergraph with the vertex set $V=\{1,2,\dots,l(m-1)\}$ and the edge set 
$\{\{i(m-1)+1,\dots,i(m-1)+m\}:i=0,1,\dots,l-2\}\cup{\{(l-1)(m-1),(l-1)(m-1)+1,\dots,l(m-1),1\}}.$ We describe it by $C_L(l)=v_1e_1v_2e_2\dots v_le_lv_1$. Here we also say that $v_1,v_2,\dots,v_l$ as the core vertices of $C_L(l)$ and the remaining vertices as the loose vertices. The constructions clearly show that a loose path and a loose cycle are linear hypergraphs. In this work, we only consider linear hypergraphs. 
A linear hypergraph $\H(V, E)$ is said to be $r$-cyclic if it contains $r$ loose cycles. For $r=1,2$ and $3$, $\H$ is unicyclic, bicyclic, and tricyclic, respectively. For $r=0$, $\H$ is called acyclic. Connected acyclic hypergraphs are known as supertrees. A supertree is called a hypertree if all the edges contain at most two non-pendant vertices.\\
A graph, which is a $2$-uniform hypergraph, is well represented by an adjacency matrix (\cite{Bapat2010}). The maximum modulus of the eigenvalues of the adjacency matrix of a graph is called the spectral radius of that graph.
Many studies have been done on the spectral radius of unicyclic graphs and trees (\cite{H1997} - \cite{Guo 2006}). In recent years, the study of the spectral radius of hypergraphs has attracted the interest of researchers (\cite{LiQi2016}-\cite{Wang2018}). Some articles examine the spectral radius of hypertrees and unicyclic hypergraphs using tensors. Hypergraphs can also be represented by adjacency matrices (see \cite {Rodriguez2002, AB2017, ABSP2023} for different representations). In \cite{Zhou2007}, using the adjacency matrix defined in \cite {Rodriguez2002}, the authors found the hypertrees with the first three largest spectral radii and the unicyclic hypergraph with the largest spectral radius for a fixed cycle length. The technique used in \cite{Zhou2007} is different from ours.\\
We make extensive use of attaching operations in this article. For instance, we attach an edge  $e$ to a vertex $v$ of a hypergraph $H$, i.e., we join $e$ and $H$ via identifying a vertex in $e$ with the vertex $v$. Similarly, we attach a path $P_L$ to a vertex $v$ of a hypergraph $H$ when we can join $P_L$ and $H$ via identifying a pendant vertex of a pendant edge of $P_L$ with the vertex $v$. We join two hypergraphs $H_1$ and $H_2$ via identifying a vertex in $H_1$ with a vertex in $H_2$ for attaching $H_1$ with  $H_2$.

\begin{definition}[\textbf{Rayleigh Quotient}]\cite{Horn}
Let $A$ be a symmetric matrix of order $n$. Then the spectral radius $\l_1(A)$ of the matrix $A,$ is given by 
$$\l_1(A)=\text{Sup}\left\{X^tAX|X\in \mathbb{R}^n, ||X||=1\right\}.$$
\end{definition}
\begin{definition}[\textbf{Reducible matrix}]\cite{Horn}
A matrix $A\in M_n$ is reducible if there is a permutation matrix  $P\in M_n$ such that 
$$
P A {P}^t=
\begin{bmatrix}
	B&C\\
	0_{n-r,r}&D
	\end{bmatrix}
and \text{ $1\leq r\leq n-1.$}
$$	
\end{definition}
A matrix $A\in M_n$ is irreducible if it is not reducible.
\begin{lemma}[\textbf{Perron-Frobenius Theorem}, Chapter 8, \cite{Horn} ]\label{th0} Let $A$ be a non-negative irreducible matrix. Then
	\begin{itemize}
		\item $A$ has a positive eigenvalue $\lambda$ with positive eigenvector.
		\item{$\lambda$ is simple} and for any other eigenvalue $\mu$ of $A$,~$|\mu|\leq{\lambda}.$
	\end{itemize}
\end{lemma}
This positive eigenvalue and the corresponding eigenvector are known as Perron eigenvalue and Perron eigenvector, respectively. When $A$ is symmetric, only the Perron eigenvector is positive.

\begin{lemma}[\textbf{Schur complement}, \cite{Bapat2010}]\label{lemma0}
	Let $A$ be an $n\times{n}$ matrix partitioned as 
	$$
	A=
	\begin{bmatrix}
		A_{11}&A_{12}\\
		A_{21}&A_{22}
	\end{bmatrix},
	$$
	where $A_{11}$ and $A_{22}$ are square matrices. If $A_{11}$ and $A_{22}$ are invertible, then
	\begin{align}
		det(A)&=det(A_{11})det(A_{22}-A_{12}A^{-1}_{11}a_{21})\\
		&=det(A_{22})det(A_{11}-A_{21}A^{-1}_{11}A_{12}).
	\end{align}
\end{lemma}
Now we recall the \textit{Adjacency matrix of hypergraphs} from \cite{AB2017}.
	Let $\H(V,E)$ be a hypergraph with the vertex set $V=\{1,2,\dots ,n\},$ the edge set $E.$ The $(i,j)$-th entry of the adjacency matrix $A_{\H}$ of $\H$ is defined as 
	$$
	(A_{\H})_{ij}=
	\begin{cases}
		\sum\limits_{e\in{E};i,j\in{e}}\dfrac{1}{|e|-1}&\text{if $i\sim j$,}\\
		0&\text {otherwise.}
	\end{cases}
	$$
	The degree of a vertex $i$ of $H$ is given by $d_i=\sum\limits_{j\in V}(A_{\H})_{ij}.$ 
	 Here we consider linear $m$-uniform hypergraphs. So we have
	 $$(A_{\H})_{ij}=\dfrac{1}{m-1}
	 \begin{cases}
	 1&\text{if $i\sim j$,}\\
	 0&\text{ otherwise. }
	 \end{cases}
     $$
     We denote the spectral radius of the adjacency matrix $A_{\H}$, by spectral radius of $\H$ i.e., by $\l_1(\H)$ and the Perron eigenvector of $A_{\H}$ by Perron eigenvector of $\H.$ 
   
 The results proved here are also true for the adjacency matrix defined in \cite{Rodriguez2002}\footnote{In 2002 Rodriguez defined the adjacency matrix and corresponding Laplacian matrix for uniform hypergraphs and which is the matrix $(m-1)A_{\H}$.}

Now we recall two hypergraph operations from \cite{LiQi2016}

	\textbf{\textit{Edge moving operation}:} Let $\H(V,E)$ be an $m$- uniform hypergraph with $u\in V$ and $e_1,e_2,\dots,e_s\in E$, such that, $u\notin e_i$ for $i=1,2,\dots,s.$ Also suppose that $u_i\in e_i.$ Construct $e'_i=(e_i\backslash \{u_i\})\cup \{u\}$. Let $\H^*=(V,E^*)$ be the hypergraph with $E^*=(E\setminus \{e_1,e_2,\dots,e_s\})\cup \{e'_1,e'_2,\dots,e'_s\}.$ Then we say that $\H^*$ is obtained from $\H$ by moving the edges $\{e_1,e_2,\dots,e_s\}$ from $\{u_1,u_2,\dots,u_s\}$ to $u.$

\textbf{\textit{Edge-releasing operation}:} Let $\H(V,E)$ be an m-uniform hypergraph and $e$ be a non-pendant edge with $u\in{e}$. Let $\{e_1,e_2,\dots,e_s\}$ be all the edges of $\H$ adjacent to the edge $e$. Also let $e\cap e_i=\{v_i\}$ for $i=1,2,\dots,s.$ Let  $\H^*$ be the hypergraph obtained from $\H$ by moving edges $\{e_1,e_2,\dots,e_s\}$ from $\{v_1,v_2,\dots,v_s\}$ to $u.$ Then $\H^*$ is said to be obtained form $\H$ by an edge releasing operation on $e$ at $u$ and we write $\H^*=\H^{r}(e;u).$
We write $\H^{r}(e),$ to denote the hypergraph $\H^{r}(e;u)$, when $u\in{e}$ and $X_{u}=$max$\{X_v|v\in{e}\}$ for the Perron eigenvector $X$ of $\H$.

Now, we define a new hypergraph operation, edge spreading, in the following definition

\begin{definition}\sloppy
	$\textbf{ Edge Spreading :}$\\
Let $\H(V,E)$ be an $m$-uniform hypergraph and $u_1,u_2,\dots,u_p\in V.$ For $s=1,2,\dots,p$, let $E(u_s)=\{e_{s1},e_{s2},\dots,e_{sr_s}\}\subset E$ be such that $u_s\in e_{st}$ for all $t=1,2,\dots,r_{s}$ and $V_s=\{v_{s1},v_{s2},\dots,v_{sr_s}\}\subset V$ be such that $v_{st}\notin {e_{st}}$ for all $s=1,2,\dots,p$; $t=1,2,\dots,r_s$. Also for all $s,t$ let $e'_{st}=(e_{st}\backslash{\{u_s\}})\cup \{v_{st}\}$. Let $\H^{ES}(V,E^{ES})$ be the hypergraph with $E^{ES}=(E\backslash{\{e_{st}|s=1,2\dots,p,t=1,2,\dots,r_s\}})\cup\{e'_{st}|s=1,2\dots,p,t=1,2,\dots,r_s\}.$ Then we say that $\H^{ES}$ is obtained from $\H$ by spreading the edges in $E(u_1),E(u_2),\dots,E(u_p)$ from $u_1,u_2,\dots,u_p$ over $V_1,V_2,\dots,V_p$ and we denote it by\\ $\H^{ES}=\H \big[E(u_1),E(u_2),\dots,E(u_p);V_1,V_2,\dots,V_p\big].$
We consider the edge spreading for which $\H^{ES}$ is simple, $m$-uniform.
\end{definition}

In this definition if we take $r_s=1,$ and $V_s=\{u\},$ for all $s,$ then its becomes the edge moving operation. Now we have the following theorem on edge spreading operation.

\begin{theorem}\label{ES1}
	Let $\H^{ES}=\H \left[E(u_1),E(u_2),\dots,E(u_p);V_1,V_2,\dots,V_p\right]$ and $X$ be the Perron eigenvector of $\H$. Also let $X$ satisfies the following conditions

\begin{enumerate}[$(A)$]
 \item $\sum_{v\in V_s}X_v\geq |E(u_s)|X_{u_s}$ when the edges in $E(u_s)$ are all pendant at the vertex $u_s$, and\\
\item  $X_v\geq X_{u_s}$ for all $v\in{V_s}$ if edges in $E(u_s)$ are not all pendant.
\end{enumerate}
Then $\l_1(\H^{ES})>\l_1(\H).$
\end{theorem}
\begin{proof}
	Let  $E_s=\cup_{s=1}^pE(u_s),E_{11}=\cup\{E(u_s)|e\in E(u_s) \implies e \text{ is pendant}\},E_{22}=E_s\backslash{E_{11}}.$
	\begin{align*}
		\l_1(\H^{ES})-\l_1(\H)&\geq  X^tA_{\H^{ES}}X-\l_1(A_{\H}) \\
		&= X^tA_{\H^{ES}}X-X^tA_{\H}X \\
		&=\dfrac{1}{m-1}\left[\sum_{i\sim j ~\text{in}~ \H^{ES}}2X_iX_j-\sum_{i\sim j~ \text{in}~ \H}2X_iX_j\right] \\
		&=\dfrac{2}{m-1}\sum\limits_{s=1;E(u_s)\subset E_{11}}^p\left[\sum\limits_{j\in e; e\in E(u_s);j\neq u_s}\sum\limits_{v\in V_s}X_vX_j-\sum\limits_{j\in e;e\in E(u_s);j\neq u_s}\sum\limits_{v\in V_s}X_{u_s}X_j\right]  \\
		 & \ \ \ \ +\dfrac{2}{m-1}\sum\limits_{s=1;E(u_s)\subset E_{22}}^p\left[\sum\limits_{j\in e; e\in E(u_s);j\neq u_s}\sum\limits_{v\in V_s}X_vX_j-\sum\limits_{j\in e;e\in E(u_s);j\neq u_s}\sum\limits_{v\in V_s}X_{u_s}X_j\right] \\
		&=\dfrac{2}{m-1}\sum\limits_{s=1;E(u_s)\subset E_{11}}^p\left[\sum\limits_{j\in e; e\in E(u_s);j\neq u_s}\sum\limits_{v\in V_s}X_vX_j-|E(u_s)|X_{u_s}\sum\limits_{j\in e;e\in E(u_s);j\neq u_s}X_j\right] \\
		&\ \ \ \ +\dfrac{2}{m-1}\sum\limits_{s=1;E(u_s)\subset E_{22}}^p\sum\limits_{j\in e; e\in E(u_s);j\neq u_s}\sum\limits_{v\in V_s}\left(X_v-X_{u_s}\right)X_j \\
		&=\dfrac{2}{m-1}\sum\limits_{s=1;E(u_s)\subset E_{11}}^p\left(\sum\limits_{v\in V_s}X_v-|E(u_s)|X_{u_s}\right)\sum\limits_{j\in e; e\in E(u_s);j\neq u_s}X_j  \\
		&\ \ \ \ + \sum\limits_{s=1;E(u_s)\subset E_{22}}^p\sum\limits_{j\in e; e\in E(u_s);j\neq u_s}\sum\limits_{v\in V_s}\left(X_v-X_{u_s}\right)X_j \\
		&\geq{0}.
	\end{align*}
Now we prove that $\l_1(\H^{ES})>\l_1(\H)$.
	Here we have $X_v\geq X_{u_s}$ for all $v\in V_s$ or $\sum\limits_{v\in V_s}X_v\geq |E(u_s)|X_{u_s}$. Let $\sum\limits_{v\in V_s}X_v\geq |E(u_s)|X_{u_s}$. Then $|E(u_s)|=|V_s|$ and $X>0$ implies that $X_v\geq X_{u_s}$ for some $v$ say $X_v\geq X_{u_s}.$ Therefore, in the both cases, we have a vertex $v\in V_s,$ such that $X_v\geq X_{u_s}$.
	Now if $\l_1(\H^{ES})=\l_1(\H)$ then from the above equation we have 
	\begin{align*}
		\l_1(\H)X_v&=\l_1(\H^{ES})X_v\\
		&=\dfrac{1}{m-1}\sum_{j;j\sim v ~\text{in}~ \H^{ES}}X_j\\
		&=\dfrac{1}{m-1}\left[\sum_{j; j\sim v ~\text{in}~ \H}X_j+\sum_{j;j\in e,e\in E(u_s),v\in e' \text{ in $\H^{ES}$};j\neq u_s}X_j \right]\\
		&=\l_1(A_{\H})X_v+\dfrac{1}{m-1}\sum_{j;j\in e,e\in E(u_s),v\in e' \text{ in $\H^s$};j\neq u_s}X_j\\
		&>\l_1(\H)X_v\\
	\implies \l_1(\H)&>\l_1(\H), \ \ \ \ \text{which is not possible.}	
	\end{align*}
Therefore we have  $\l_1(\H^{ES})>\l_1(\H)$ and this completes the proof.
\end{proof}

\begin{corollary} \label{Lemma 1}
	Let $\H^*$ be the hypergraph obtained from $\H$ by moving the edges $e_1,e_2,\dots,e_s$ from $u_1,u_2,\dots,u_s$ to $u$. Also let $X$ be the Perron eigenvector of $A_{\H}$ and $X_u\geq max\{X_{u_i}| i=1,2,\dots,s\}$. If $\H^*$ is simple and m-uniform, then $\l_1(\H^*)>\l_1(\H)$.
\end{corollary}
\begin{proof}
The proof follows from the theorem \ref{ES1} by taking $E(u_s)=\{e_s\}$ and $V_s=\{u\}$ for all $s=1,2,\dots p.$.
\end{proof}

\begin{corollary} \label{Lemma 2}
	Let $e$ be non-pendant edge of $\H$ and $u\in {e}$. Then $\l_1(\H^r(e;u))>\l_1(\H).$
\end{corollary}
\begin{proof}
	Note that $\H^{r}(e;u)$ and $\H^{r}(e;v)$ are isomorphic for any two $u,v$. Again for any Perron eigenvector X of $A_{\H}$ either $X_u\geq X_v$ or $X_u<X_v$.
\end{proof}
Now we recall equitable partition for hypergraphs from \cite{ASAB2020},\\
\textbf{\textit{Equitable Partition for Hypergraphs:}}\\
	Let $\mathcal{H}(V,E)$ be an $m$-uniform hypergraph with adjacency matrix $A_{\H}$. We say a partition $\pi=\{C_1,C_2,\dots,C_k\}$ of $V$ is an equitable partition of $V$
	if for any $p,q\in\{1,2, \dots ,k\}$ and for any $i\in{C_p}$,~$$\sum\limits_{j,j\in{C_q}}(A_\mathcal{H})_{ij}=b_{pq},$$ where $b_{pq}$ is a constant depends
	only on $p$ and $q$.\\
	 Note that any two vertices belonging to the same part of the partition have the same degree. For an equitable partition with 
	$k$-number of parts, we define the \textit{quotient matrix} $B$ as $(B)_{pq}=b_{pq}$, for $1\leq{p,q}\leq{k}.$
	
 The quotient matrices are diagonalizable and the eigenvalues of $B$ are also the eigenvalues of $A_{\H}$ with same multiplicity \cite{ASAB2020}.
\begin{proposition}
	Let $\H$ be a connected hypergraph. Also let $\pi$ be an equitable partition of $\H.$ Then the
	spectral radius of $A_{\H},$ is the spectral radius of the corresponding quotient matrix $B.$
\end{proposition}
\begin{proof}
	Let  $\pi=\{V_1,V_2,\dots,V_l,V_{l+1},V_{l+2},\dots,V_{l+p}\}$ be an equitable partition of $\H(V,E).$
	First we note that $B$ is non-negative. Now we claim that $B$ is irreducicle also. If not, then the quotient matrix $B$ is reducible, i.e., we can arrange the parts of the partition  as follows 
	$$
	B=
	\begin{bmatrix}
		B_{11}&\mathbb{O}\\
		\mathbb{O}&B_{22}
	\end{bmatrix},
	$$ 
	where $\mathbb{O}$ is the zero matrix, and since $(B)_{ij}=0,$ if and only if $(B)_{ji}=0$. Let $U_1=\{V_1,V_2,\dots,V_l\}$ and $U_2=\{V_{l+1},V_{l+2},\dots,V_{l+p}\}$ be such an arrangement. Then we have 
	$$
	(B)_{ij}=
	\begin{cases}
	0&\text{ for all $i\in \{1,2,\dots,l\}$,~$j\in\{l+1,l+2,\dots,l+p\}$ and}	\\
	0&\text{ for all $j\in \{1,2,\dots,l\}$,~$i\in\{l+1,l+2,\dots,l+p\}$}
	\end{cases}
	$$
	Again $(B)_{ij}=0\implies (A_{\H})_{st}=0$ for all $s\in{V_i}$ and $t\in{V_j}.$
	Let $W_1=\cup_{i=1}^lV_i$ and $W_2=\cup_{i=l+1}^{l+p}V_i$. Then we can write 
	$$
	A_{\H}=
	\begin{bmatrix}
		A_{11}&\mathbb{O}\\
		\mathbb{O}&A_{22}
	\end{bmatrix}
	$$
	where $(A_{\H})_{ij}=0$ for all $i\in W_1,j\in W_2$ and for all $j\in W_1,i\in W_2$ also. This shows that $A_{\H}$ is also reducible, and which is a contradiction. Hence we have the \textit{quotient matrix} $B$ is irreducible. Now by Perron-Frobenius theorem, $B$ has a positive eigenvector corresponding to the spectral radius of $B,$ which implies that, the spectral radius of $B$ is an eigenvalue of $A_{\H}$ with a positive eigenvector. This completes the proof.
\end{proof}

Let $G(V,E)$ be a graph. The m-uniform hypergraph $\H$ is said to be m-th \textit{power hypergraph of $G$} if $V(\H)=V\cup_{e\in E}\{v_{e,1}v_{e,2},\dots,v_{e,m-2}\}$ and $E(\H)=\{e\cup\{v_{e,1}v_{e,2},\dots,v_{e,m-2}\}| e\in{E}\}.$ Also we write $\H=\P(G).$ Let $\mathcal{P(\H)}$ be the set of m-th \textit{power hypergraphs} of graphs.  \\
Now for an $\H$ m-th \textit{power hypergraph} of a graph $G(V_g,E_g),$ 
let $V_g=\{v_1,v_2,\dots,v_n\}$ and
$$
V_e'=
\begin{cases}
	\{v_{e,1},v_{e,2},\dots,v_{e,m-2}\}\ \ \ \  &\text{if $e$ is an non-pendant edge in $G$,}\\
	\{v_{e,1},v_{e,2},\dots,v_{e,m-2}\}\cup \{u\}\ \ \ \  &\text{if $e$ is a pendant edge in $G$ and u is the pendant vertex in e. }
\end{cases}
$$ 
\sloppy{
Let $U_1=\{ \{v_i\}| v_i \text{ is non pendant vertex in G}\}$,~ $E_g'=\{e\in E_g|e \text{ is non-pendant in $G$}\},$~ $E_g''=\{e\in E_g| \text{e is pendant edge in G}\}$. For $e\in E_g''$, let $V_e''=\left(\cup_{f\in E_g'';|e\cap f|=1}V'_f\right)\cup V_e'.$ Let $U_2=\{V_e'|e\in{E_g'}\}$ and $U_3=\{V_e''|e\in{E_g''}\}.$}
Then  ${\pi}_p=U_1\cup U_2\cup U_3,$ forms an equitable partition for the hypergraph $\H=\P(G)$. 
From now, by the \textit{quotient matrix} of a hypergraph $\H\in \PH$ we mean the \textit{quotient matrix} corresponding to this equitable partition, if the partition is not specified.

\section{Hypertrees with largest and second-largest spectral radii with given diameter}

Let $\T(n;m)$ be the set of all m-uniform hypertree with $n$ vertices and
 $\T(n;m;d)=\{T\in \T(n;m)| \text{ diameter of $T$ is $d$}\}.$
We note that for any $T\in \T(n;m)$,~ $|E(T)|=(n-1)/(m-1).$
Let $T_d(c_2,c_3,\dots,c_d)$ be the hypertree obtained by attaching $c_2,c_3,\dots,c_d$ numbers pendant edges to the core vertices $v_2,v_3,\dots,v_d$ (respectively) of the path $P_L(d)=v_1e_1v_2e_2\dots v_{d-1}e_{d-1}v_de_dv_{d+1}$. We write $T_d(c_{i_1},c_{i_2},\dots,c_{i_r})$ to denote the hypertree $T_d(c_2,c_3,\dots,c_d)$ when $c_j=0$ for $i\neq i_1,i_2,\dots,i_r$.
Now from the definition of power hypergraph, we can write $T_d(c_p=l_p)=\P(T_g)$ where $T_g$ is the graph with the vertex set $V_g=\{1,2,3,\dots,p-1,p,p_1,p_2,\dots,p_{l_p},p+1,\dots,d,d+1\}$ in which vertex $i$ is adjacent to the vertex $i+1$ for $i=1,2,\dots,d$ and $p_j$'s are adjacent to the vertex $p$. The hypertree $T_2(c_2=k-2)$ is known as the hyperstar, and the non-pendant vertex as the center of the hyperstar.

\subsection{Largest spectral radius}

Let $\T(n;m;d;s)=\{T\in{\T(n;m;d)}| $T contains $s$ vertices of degree greater than $2\}.$

\begin{lemma}\label{lemma B}
	Let $T\in{\T(n;m;d;s)}$ and $s>1$. Then there exists $T^*\in{\T(n;m;d;1)}$ such that $\l_1(A_{T^*})> \l_1(A_{T})$. 
\end{lemma}	

\begin{proof}
	Let $P_L(d)=v_1e_1v_2e_2v_3\dots v_de_dv_{d+1}$ be a path of length $d$, in $T$. For the Perron eigenvector $X$ of $A_T,$ we choose $v\in \{v_1,v_2,\dots,v_d\}$ such that $X_v=max\{X_{v_2},\dots,X_{v_d}\}$. Now let $T^*$ be the hypertree obtained from $T$ by moving the edges (not from $P_L(d)$) that contain $v_i$ from $\{v_2,\dots,v_d\}\setminus \{v\}$ to $v$. Then $T^*\in \T(n;m;d;1)$ and $\l_1(T)<\l_1(T^*).$
\end{proof}

\begin{lemma}\label{lemma C}
	Let $T\in{\T(n;m;d;1)}$ and $d>2$. Then there exists $T^*\in{\T(n;m;d-1;1)}$ such that $\l_1(A_{T^*})> \l_1(A_{T})$.
\end{lemma}
\begin{proof}
	Release any non-pendant edge of $T$.
\end{proof}

We have the following theorem

\begin{theorem} \label{Th1}
	Let $T^*_i\in {\T(n;m;i;1)}$ be such that $\l_1(T^*_i)=max \Big\{\l_1(T)| T\in {\T(n;m;i;1)}\Big\}$ for $i=2,3,\dots,d-1$. Then
	\begin{enumerate}[$(i)$]
		\item  $\l_1(T^*_2)> \l_1(T^*_3)>\l_1(T^*_4)>\dots >\l_1(T^*_{d-1})>\l_1(A_{P_L(d)}).$
		
		\item $T^*_i=T_i(c_p=k-i)$ for some $p=2,3,\dots,\Big[\dfrac{i}{2}\Big]+1.$
	\end{enumerate}
where $k=(n-1)/{(m-1)}$ and $d\geq 2.$ 
\end{theorem}
\begin{proof}
	\begin{enumerate}[(i)]
		\item Follows from lemma $\ref{lemma C}$.
		\item Using lemma $\ref{lemma B}$, we have $T^*_i\in{\T(n;m;i;1)}$. Let $P_L(i)=v_1e_1v_2e_2\dots v_ie_iv_ie_i,v_{i+1}$ be a path of length $i$, in $T_i^*$, and the degree of the vertex $v_p(p\neq 1)$ is greater than two. Now, release all the non pendant edges of $T^*_i$ containing $v_p$, but not in $P_L$, and this completes the proof.
	\end{enumerate}
\end{proof}

Now we have a corollary of the above theorem,
\begin{corollary}\label{hyperstar}
	Let $T$ be an $m$-uniform hypertree, then $\l_1(T)\leq \left(m-2+\sqrt{(m-2)^2}+4k(m-1)\right)/(2m-2)$, and the equality holds if and only if $T$ is the hyperstar $T_2(c_2=k-2).$
\end{corollary}

\begin{lemma}\label{lemma3A}
	Let $\H(V,E)$ be an $m$-uniform hypergraph. Let $\H_1$ be the hypergraph obtained by attaching a loose-path $P_L(l)=v_1e_1v_2e_2v_3e_3\dots v_le_lp(=v_{l+1})$ to $\H$ at the vertex $p$.
	If $$\l_1(\H_1)\geq \dfrac{m-1+\sqrt{(m-1)(m+7)}}{2(m-1)},$$ then for the Perron eigenvector $X$, of $\H_1$, we have $1>{X_i}/X_{i+1}>X_{i-1}/X_{i}$ for $i=2,3,\dots,l-1$ where $X_{i}=X_{v_i}$ for $i=1,2,\dots,l$ and $X_{l+1}=X_p.$
\end{lemma}	

\begin{proof}
	
	Let $V_{1,2}=e_1\backslash \{v_2\}$ and $V_{i,i+1}=e_i\backslash \{v_i,v_{i+1}\}$, for $i=2,3,\dots,l.$
	Then $\pi=\{\{v\}|v\in{V\backslash\{p\}}\}\cup \{\{v_2\},\{v_3\}, \dots\{v_l\},V_{1,2},V_{2,3},\dots,V_{l,l+1}\}$ forms an equitable partition of $\H_1.$ So the largest eigenvalue of the quotient matrix is the spectral radius of the hypertree $\H_1,$ and the Perron eigenvector of $\H_1$ has the same value at each vertex within every part of the partition. Now let $X$ be the Perron eigenvector of $\H_1$ corresponding to the spectral radius $\l_1$. Let $\l=(m-1)\cdot\l_1$ and $a_{i,i+1}=X_i/X_{i+1}$ for $i=1,2,3,\dots,l$. Now for any $v\in{V_{1,2}}$ we have 
\begin{align*}
	\implies & (\l-m+2)X_v=X_2.
\end{align*}
Considering the eigenvalue equation of $A_{\H_1}$ for the eigen-pair $(\l_1,X)$ we have 
\begin{align*}
	\l X_2= \dfrac{m-1}{\l-m+2}X_2+\dfrac{m-2}{\l-m+3}(X_2+X_3)+X_3\\
	\implies (\l-\dfrac{m-1}{\l-m+2}-\dfrac{m-2}{\l-m+3})X_2=\dfrac{\l+1}{\l-m+3}X_3.
\end{align*}
Let $f_{2,3}(\l)=\l-\dfrac{m-1}{\l-m+2}-\dfrac{m-2}{\l-m+3}$,~$g(\l)=\dfrac{\l+1}{\l-m+3}$, and $a_{2,3}=X_2/X_3.$ 
Then 
\begin{align*}
	f_{2,3}(\l)-g(\l)&=\l-\dfrac{m-1}{\l-m+2}-\dfrac{m-2}{\l-m+3}-\dfrac{\l+1}{\l-m+3}\\
	&=\dfrac{(\l-m+2)\left(\l^2-(m-2)\l-3(m-1)\right)+(\l-m+1)(m-1)}{(\l-m+2)(\l-m+3)}
\end{align*}

Since $\l^2-(m-2)\l-3(m-1)$ is an increasing function for $\l>(m-2)/2$ and here $\l>\left(m-1+\sqrt{(m-1)(m+7)}\right)/2$ it follows that $f_{2,3}(\l)-g(\l)>0$ and hence $a_{2,3}=X_2/X_3<1.$
Also for any $i=3,4,\dots,l$ we have 

\begin{align}\label{Eq PV1}
	&\l X_i= \dfrac{m-2}{\l-m+3}(X_{i-1}+X_i)+\dfrac{m-2}{\l-m+3}(X_i+X_{i+1})+X_{i-1}+X_{i+1}\notag \\
	\implies &(\l-\dfrac{2(m-2)}{\l-m+3})X_i-\dfrac{\l+1}{\l-m+3}X_{i-1}=\dfrac{\l+1}{\l-m+3}X_{i+1}
\end{align}
Again $X_2<X_3$, so when $i=3$ we have
\begin{align*}
	&(\l-\dfrac{2(m-2)}{\l-m+3})X_3-\dfrac{\l+1}{\l-m+3}X_3<\dfrac{\l+1}{\l-m+3}X_{4}\\
	\implies &(\l-\dfrac{\l+2m-3}{\l-m+3})X_3<\dfrac{\l+1}{\l-m+3}X_4
\end{align*}
Let $f_{3,4}(\l)=\l-\dfrac{\l+2m-3}{\l-m+3}$,~$g(\l)=\dfrac{\l+1}{\l-m+3}$.
Then 
\begin{align*}
	f_{3,4}(\l)-g(\l)&=\l-\dfrac{\l+2m-3}{\l-m+3}-\dfrac{\l+1}{\l-m+3}\\
	&=\dfrac{\l^2-(m-1)\l-2(m-1)}{\l-m+3}
\end{align*}
So for $\l \geq \left(m-1+\sqrt{(m-1)(m+7)}\right)/2$ we get $f_{3,4}(\l)-g(\l)>0$ and hence $a_{3,4}(\l)=X_3/X_4<1$. \\
Using this from equation $(\ref{Eq PV1})$, we have $X_i<X_{i+1}$ for $i=4,\dots,l$.

Now from \eqref{Eq PV1} for $i=3,4,\dots,p-1$, we have
\begin{align}\label{Eq PV2}
	a_{i,i+1}=\dfrac{g(\l)}{\l-\dfrac{2(m-2)}{\l-m+3}-\dfrac{\l+1}{\l-m+3}a_{i-1,i}}
\end{align}

Again 
\begin{align*}
	a_{2,3}-a_{1,2}&=\dfrac{g(\l)}{\l-\dfrac{m-1}{\l-m+2}-\dfrac{m-2}{\l-m+3}}-\dfrac{1}{\l-m+2}\\
	&=\dfrac{(\l-m+3)(m-1)}{(\l-m+2)\left((\l-m+3)(\l^2-m\l+2\l-2m+3)+m-2\right)}\\
	&>0 \text{  since $\l>m-1$}\\
	\implies a_{2,3}&>a_{1,2}.
\end{align*}
Also
\begin{align}\label{Eq PV3}
	a_{3,4}-a_{2,3}&=\dfrac{g(\l)}{\l-\dfrac{2(m-2)}{\l-m+3}-\dfrac{\l+1}{\l-m+3}a_{2,3}}-\dfrac{g(\l)}{\l-\dfrac{m-1}{\l-m+2}-\dfrac{m-2}{\l-m+3}}
\end{align}
Now 
\begin{align*}
	&\left(\l-\dfrac{m-1}{\l-m+2}-\dfrac{m-2}{\l-m+3}\right)-\left(\l-\dfrac{2(m-2)}{\l-m+3}-\dfrac{\l+1}{\l-m+3}a_{2,3}\right)\\
	&=\dfrac{\l+1}{\l-m+3}( a_{2,3}-a_{1,2})\\
	&>0
\end{align*}
	So from equation \eqref{Eq PV3} we have $a_{3,4}>a_{2,3}$ and hence from equation \eqref{Eq PV2} we have $1>a_{i,i+1}>a_{i-1,i}$ for $i=2,3,\dots,p-1.$

\end{proof}

\begin{lemma}\label{lemma3B}
	Let $X$ be the  Perron eigenvector corresponding to the eigenvalue $\l_1(T_d(c_p=k-d))$ for some $p=2,3,\dots,d$. Also let $(k-d-1)(m-1)\geq 6$ and $a_{i,i+1}=X_i/X_{i+1}$ for $i=1,2,3,\dots,p-1$. Then $1>a_{i,i+1}>a_{i-1,i}$ for $i=2,3,\dots,p-1$ and  $a_{i,i+1}<a_{i-1,i}<1$ for $i=p+1,p+2,\dots,d+1.$
\end{lemma}
\begin{proof}
	Note that we can consider $T_d(c_p=k-d)$ as the hypergraph obtained by attaching two loose-paths $P_1=v_1e_1v_2e_2\dots v_pe_pv_{p+1}$, and $P_2=u_1f_1u_2f_2\dots u_{d-p}f_{d-p}u_{d-p+1}$ at the center of the hyperstar $T_2(k-d).$
	Now 
	\begin{align*}
		\l_1\left(T_d(c_p=k-d)\right)>\l_1\left(T_2(k-d+2)\right)&=\dfrac{m-2+\sqrt{(m-2)^2+4(k-d+2)(m-1)}}{2(m-1)}\\
		&=\dfrac{m-2+\sqrt{(m+4)^2+4\left((k-d-1)(m-1)-6\right)}}{2(m-1)}.
	\end{align*}
	So for $(k-d-1)(m-1)\geq 6$ we have $\l_1\big(T_d(c_p=k-d)\big)>\dfrac{m+1}{m-1}$. On the other hand using A.M $\geq$ G.M, we have 
	$\dfrac{m-1+\sqrt{(m-1)(m+7)}}{2(m-1)}< \dfrac{m-1+\dfrac{m-1+m+7}{2}}{2(m-1)}=\dfrac{m+1}{m-1}$. Therefore using lemma \ref {lemma3A} we have the required result.
\end{proof}

\begin{proposition} \label{Th2}
	Let $T$ be an $m$-uniform hypertree with  $k$ edges and diameter $d$, where $(k-d-1)(m-1)\geq 6$. Then the spectral radius of $T$,~ $\l_1(T)<\left(m+\sqrt{m^2+4(k-d+4)(m-1)}\right)/(2m-2).$
\end{proposition}
\begin{proof}
	Let $a=k-d$ and $(\l_1,X)$ be the Perron eigen-pair of $A_{T_d(c_p=a)}$. Also let $\l=(m-1)\l_1$. Then we have 
	\begin{align}\label{Eq1}
		&\left(\l-\dfrac{a(m-1)}{\l-m+2}-\dfrac{2(m-2)}{\l-m+3}\right)X_p-\dfrac{\l+1}{\l-m+3}X_{p-1}=\dfrac{\l+1}{\l-m+3}X_{p+1} \notag \\
		\implies& \left(\l-\dfrac{a(m-1)}{\l-m+2}-\dfrac{2(m-2)}{\l-m+3}\right)X_p-\dfrac{\l+1}{\l-m+3}X_p<\dfrac{\l+1}{\l-m+3}X_{p+1} \notag \\
		\implies& \left(\l-\dfrac{a(m-1)}{\l-m+2}-\dfrac{\l+2m-3}{\l-m+3}\right)X_p<\dfrac{\l+1}{\l-m+3}X_{p+1}  
	\end{align}
	Again from lemma $\ref{lemma3B}$ we have $X_p>X_{p+1}$. So from the above inequality, we have 
	\begin{align*}
		&\l-\dfrac{a(m-1)}{\l-m+2}-\dfrac{\l+2m-3}{\l-m+3}<\dfrac{\l+1}{\l-m+3}\\
		\implies& \l<\left(m+\sqrt{m^2+4(k-d+4)(m-1)}\right)/2.
	\end{align*}

\end{proof}

\begin{lemma}\label{lemma 4}
	Let $d\geq 5,2p<d$, and $(k-d-1)(m-1)\geq 6$. Also let $X$ be the Perron eigenvector of the hypertree $T_d(c_p=k-d).$ Then for any $v\in e_1$,~$X_v> X_d.$ 
\end{lemma}

\begin{proof}
	
		Let $\l=(m-1)\l_1,a=k-d,g=\dfrac{\l+1}{\l-m+3},f_1=\l-\dfrac{m-1}{\l-m+2}-\dfrac{m-2}{\l-m+3},f_2=\l-\dfrac{2(m-2)}{\l-m+3},f_3=\l-\dfrac{(a+1)(m-1)}{\l-m+2}-\dfrac{m-2}{\l-m+3},f_4=\l-\dfrac{a(m-1)}{\l-m+2}-\dfrac{m-2}{\l-m+3}$.
	\begin{itemize}
		\item \textbf{Case I:: $p=2:$} \\
		For any $v\in e_1$ we have 
		\begin{align*}
			(\l-m+2)X_v=X_2.
		\end{align*}
	and
	\begin{align*}
		&\left(\l-\dfrac{(a+1)(m-1)}{\l-m+2}-\dfrac{m-2}{\l-m+3}\right)X_2=g.X_3\\
		\implies&(\l-m+2)\left(\l-\dfrac{(a+1)(m-1)}{\l-m+2}-\dfrac{m-2}{\l-m+3}\right)X_v=g.X_3\\
		\implies&(\l-m+2)f_3 X_v=g.X_3.
	\end{align*}

Again 
 	\begin{align*}
 	\left(\l-\dfrac{m-1}{\l-m+2}-\dfrac{m-2}{\l-m+3}\right)X_d=g.X_{d-1}.
 \end{align*}
Also 	
\begin{align*}
	\left(\l-\dfrac{2(m-2)}{\l-m+3}\right)X_{d-1}=g.(X_{d-2}+X_d).
\end{align*}
So we have 
\begin{align}
		&\left(\l-\dfrac{m-1}{\l-m+2}-\dfrac{m-2}{\l-m+3}\right)\left(\l-\dfrac{2(m-2)}{\l-m+3}\right)X_d=g^2(X_d+X_{d-2})\notag \\
		\implies& \left(\left(\l-\dfrac{m-1}{\l-m+2}-\dfrac{m-2}{\l-m+3}\right)\left(\l-\dfrac{2(m-2)}{\l-m+3}\right)-g^2\right)X_d=g^2X_{d-2}\notag \\
		\implies& (f_2.f_1-g^2)X_d=g^2X_{d-2}.
\end{align}
Now 
\begin{align*}
	(f_2.f_1-g^2)-g.(\l-m+2)f_3&=f_2f_1-g^2-g.(\l-m+2)\left(f_1-\dfrac{(k-d)(m-1)}{\l-m+2}\right)\\
	&=f_1.\left(f_2-g.(\l-m+2)\right)+g.(k-d)(m-1)-g^2\\
	&=\dfrac{\l+1}{\l-m+3}(k-d)(m-1)-\dfrac{(\l+1)^2}{(\l-m+3)^2}\\
	&\ \ \ -\dfrac{m-2}{\l-m+3}\left(\l-\dfrac{m-1}{\l-m+2}-\dfrac{m-2}{\l-m+3}\right)\\
	&>\dfrac{\l+1}{\l-m+3}(k-d)(m-1)-\dfrac{(\l+1)^2}{(\l-m+3)^2} -\dfrac{m-2}{\l-m+3}\l\\
	&=\dfrac{1}{\l-m+3}\left((\l+1)\left((k-d-1)(m-1)-\dfrac{m-2}{\l-m+3}\right)+m-2 \right)\\
	&\geq 0.
\end{align*}
 So we have $X_v>X_d$.

\item \textbf{Case II:: $p\geq 3$:}\\ 
Note that $d\geq 7.$ Let $h_1=1,h_2=f_1$ and for $i=3,4,\dots,d$
\begin{align}\label{LEQ1}
	h_i=f_2h_{i-1}-g^2.h_{i-2}.
\end{align}
 Then we have 
 \begin{align}
 	h_iX_i=gh_{i-1}X_{i+1}&\ \ \ \ \text{for $i=2,3,\dots,p-1$.}
 \end{align}
Now $(\l-m+2)X_v=X_2$ and from equation \eqref{LEQ1} we have
\begin{align}
	h_{p-1}X_2=g^{p-2}X_p.
\end{align}
Again
\begin{align*}
	&f_4X_p=g(X_{p-1}+X_{p+1})\\
	\implies &f_4h_{p-1}X_p-h_{p-1}gX_{p-1}=h_{p-1}gX_{p+1}\\
	\implies&(f_4h_{p-1}-g^2h_{p-2})X_p=gh_{p-1}X_{p+1}.
\end{align*}

So,
	    \begin{align}\label{LEQ2}
	    	(\l-m+2)(f_4h_{p-1}-g^2h_{p-2})X_v=g^{p-1}X_{p+1}.
        \end{align}  
Also for $i=2,3,\dots,d-p+1$ we have 
\begin{align}
	h_iX_{d-i+2}=gh_{i-1}X_{d-i+1}.
\end{align} 
From this we have
\begin{align}\label{LEQ3}
	h_{p+1}X_d=g^pX_{d-p}.
\end{align}

Now  
\begin{align}
	    &h_{p+1}-g.(\l-m+2)(f_4h_{p-1}-g^2h_{p-2})\\
		=&\{f_2h_p-g^2h_{p-1}\}-g.(\l-m+2)(f_4h_{p-1}-g^2h_{p-2})\notag \\
		=&\dfrac{1}{\l-m+3}\left(h_{p-1}\left((\l+1)(k-d)(m-1)-(\l+1)g-(m-2)f_2\right)+
		(m-2)g^2h_{p-2}\right)\notag\\
		>&\dfrac{1}{\l-m+3}\left(h_{p-1}\left((\l+1)(k-d)(m-1)-(\l+1)g-(m-2)\l \right)+(m-2)g^2h_{p-2}\right)\ \ \ \text{(since $f_2<\l$)}\notag\\
		=&\dfrac{1}{\l-m+3}\left(h_{p-1}\left((\l+1)(k-d-1)(m-1)-(\l+1)g+(\l+m-1)\right)+(m-2)g^2h_{p-2}\right).\notag\\
\end{align}
Again
\begin{align*}
&(\l+1)(k-d-1)(m-1)-(\l+1)g+(\l+m-1)\\
&=(\l+1)\left((k-d-1)(m-1)-\dfrac{m-2}{\l-m+3}\right)+m-2\\
&>0.
\end{align*} 
From the lemma $\ref{lemma3B}$, we have $X_{d-p}\leq X_{p+1}$ when $d>2p$ and so from equation $(\ref{LEQ2}),(\ref{LEQ3})$ we have $X_v>X_d.$

\end{itemize}
\end{proof}

Now we are ready to find the hypertree, with a given (odd) diameter, posseses the highest spectral radius' 
\begin{theorem}\label{Th3}
	Let $T$ be an m-uniform hypertree with $k$-edges and diameter $d=2p+1$, where $(k-d-1)(m-1)\geq 6$. Then
	 $\l_1(T)\leq \l_1(T_d(c_{p+1}=k-d))$  and the equality holds only when $T=T_d(c_{p+1}=k-d).$ 
\end{theorem}
\begin{proof}
	From the theorem $\ref{Th1}$ we have $\l_1(T)\leq \text{Max}\bigg\{\l_1(T_d(c_i=k-d)):i=2,3,\dots,\Big[\dfrac{d}{2}\Big]+1\bigg\}.$ Let $2i<d$. Then from the lemma $\ref{lemma 4}$ we have in $T_d(c_i=k-d),$ for any $v\in e_1$ $X_{v}>X_{d}$. Now we move the edge $e_{d+1}$ from $d$ to $v(v\neq 2)$. Therefore  $\l_1(T_d(c_i=k-d))<\l_1(T_d(c_{i+1}=k-d))$ and this completes the proof.
\end{proof}

Now to find the hypertree with the highest spectral radius and even diameter we prove the following lemma.
\begin{lemma}\label{lemma 5}
	Let $d=2p$ be even and $k\geq \left((4d^2-1)(m-1)+2\right)/4$. Then for the Perron eigenvector $X$ of $T_d(c_p=k-d)$ we have for any $v\in{e_1}$,~ $X_v>X_d$.
\end{lemma}

\begin{proof}
	From lemma \ref{lemma 4} we have
	 \begin{align*}
		(\l-m+2)(f_4h_{p-1}-g^2h_{p-2})X_v&=g^{p-1}X_{p+1}\\
		& \text { and} \\
	     h_pX_d&=g^{p-1}X_{p+1}.
\end{align*}
Now 
\begin{align*}
	&h_p-(\l-m+2)(f_4h_{p-1}-g^2h_{p-2})\\
	=&\left(f_2h_{p-1}-g^2h_{p-2}\right)-(\l-m+2)(f_4h_{p-1}-g^2h_{p-2})\\
	=&\left((k-d)(m-1)-(\l-m+1)f_2\right)h_{p-1}\\
	>&(k-d)(m-1)-(\l-m+1)\l, \text{ since $(f_2=\l-\dfrac{2(m-2)}{\l-m+2})$.}
\end{align*}
From corollary \ref{hyperstar} we have $\l_1(T_d(c_p=k-d))<\left(m-2+\sqrt{(m-2)^2+4k(m-1)}\right)/(2m-2)$ i.e.,
$\l<\left(m-2+\sqrt{(m-2)^2+4k(m-1)}\right)/2$ and so 
\begin{align*}
	&(k-d)(m-1)-(\l-m+1)\l \\
	>&(k-d)(m-1)-\dfrac{\sqrt{(m-2)^2+4k(m-1)}+m-2}{2}.\dfrac{\sqrt{(m-2)^2+4k(m-1)}-m}{2}\\
	=&\dfrac{1}{2}\sqrt{(m-1)(4k+m-3)}+\dfrac{m-2}{2}-d(m-1)\\
	>&0, \text{   since $k\geq \left((4d^2-1)(m-1)+2\right)/4$.}
\end{align*}
Therefore we have $X_v>X_{d}$ for any $v\in e_1.$
\end{proof}

\begin{theorem}\label{Th4}
	Let $T$ be an $m$-uniform hypertree with $k$-edges and even diameter $d=2p$. Also let 
	$k\geq \left((4d^2-1)(m-1)+2\right)/4$. Then 
	\begin{enumerate}[$(i)$]
		\item $\l_1(T)\leq \l_1\big(T_d(c_{p+1}=k-d)$ and equality holds only if $T=T_d(c_{p+1}=k-d)$
		\item $\l_1(T)<\left(m-1+\sqrt{(m-1)^2+4(k-d)(m-1)}\right)/(2m-2).$
	\end{enumerate}
\end{theorem}
\begin{proof}
 We note that if $k\geq \left((4d^2-1)(m-1)+2\right)/4$, then $(k-d-1)(m-1)\geq 6.$
	\begin{enumerate}[$(i)$]
		\item Using lemma \ref{lemma 4} we have $\l_1(T)\leq \text{Max}\left\{\l_1(T_d(c_p=k-d)),\l_1(T_d(c_{p+1}=k-d))\right\}$. Again using lemma \ref{lemma 5} we have $\l_1(T_d(c_p=k-d))<\l_1(T_d(c_{p+1}=k-d))$ and this completes the proof.
		
		\item From the proof of the lemma \ref{lemma 5}, we have 
		\begin{align*}
			(k-d)(m-1)-\l(\l-m+1)>0
			\implies \l<\left(m-1+\sqrt{(m-1)^2+4(k-d)(m-1)}\right)/2,
		\end{align*}
	and this completes the proof.
	\end{enumerate}
\end{proof}
\subsection{Second-largest spectral radius}
Now we find the hypertree with a given diameter $d$ having the second largest spectral radius.
\begin{lemma}
	Let  $T$ be an $m$-uniform hypertree with $k$-edges, diameter $d$ and $T\neq T_d(c_{\left[\dfrac{d}{2}\right]+1}=k-d)$. Then 
	\begin{enumerate}[$(i)$]
		\item 
			$\l_1(T)\leq \text{max}\left\{\l_1(T_d(c_p=1,c_{p+1}=k-d-1)),\l_1(T_d(c_{p+1}=k-d-1,c_{p+2}=1)),\l_1(T_d(c_p=k-d))\right\}$ \text{ if $d=2p+1$ and $(k-d-4)(m-1)\geq 20$.}
			
		\item 
	$\l_1(T)\leq \text{max}\{\l_1(T_d(c_{p+1}=k-d-1,c_{p+2}=1),\l_1(T_d(c_p=k-d)) \}$ \text{ if $d=2p$} and $k\geq \left((4d^2-1)(m-1)+2\right)/4.$
\end{enumerate}
\end{lemma}
\begin{proof}
\begin{enumerate}[$(i)$]
	\item Let $d=2p+1$ and $P_L(d)=v_1e_1v_2e_2\dots v_de_dv_{d+1}$ be a path of length $d$ in $T.$
	We have two cases : 
	\begin{itemize} 
		\item \textbf{Case I:: $\mathbf{d(v_{p+1})>2}$ or $\mathbf{d(v_{p+2})>2}$:} \\
		 Again we have two subcases\\
\textbf{Subcase A:: $\mathbf{d(v_i)>2}$ for some $\mathbf{v_i\neq v_{p+1},v_{p+2}}$:} \\
		For the Perron eigenvector $X$ of $T,$ let $X_{v_t}=$max$\{X_{v_i}|i\neq p+1,p+2\},$ and $X_{v_{p+1}}\geq X_{v_{p+2}}(X_{v_{p+2}}\geq X_{v_{p+1}})$. Then move the edges (not in $P_L(d)$) from $v_i's$ to $v_t$ and move edges (not in $P_L(d)$) from $v_{p+2}$ to $v_{p+1}$ ( from $v_{p+1}$ to $v_{p+2}$). Let $T'$ be the hypertree obtained from $T$ by releasing the non-pendant edges (other than the edges of $P_L(d)$). Then $\l_1(T)<\l_1(T')$. So if $T'=T_d(c_p=1,c_{p+1}=k-d-1)$ or $T'=T_d(c_{p+1}=k-d-1,c_{p+2}=1)$ then we are done. If not then we have either $T'=T_d(c_i=l(i<p+1),c_{p+1}=k-d-l)$
		or $T'=T_d(c_{p+1}=k-d-l,c_i=l(i>p+1))$ for some $l$. \\
		First we suppose that 
		 $T'=T_d(c_i(i<p+1)=l,c_{p+1}=k-d-l)$. Here for the Perron eigenvector $X'$ of $T'$ we have either $X'_i\leq X'_{p+1}$ or $X'_i>X'_{p+1}$. If $X'_i\leq X'_{p+1}$, then move all but one pendant edges from $v_i$ to $v_{p+1}$ otherwise move the pendant edges from $v_{p+1}$ to $v_i$. Let $T''$ be the resultant hypertree. Then we have $T''=T_d(c_i(i<p+1)=1,c_{p+1}=k-d-1)$ or $T''=T_d(c_i(i<p+1)=k-d)$. Again $\l_1(T_d(c_i(i<p+1)=k-d))\leq \l_1(T_d(c_p=k-d))$. So suppose $T''=T_d(c_i(i<p+1)=1,c_{p+1}=k-d-1)$. Then for the Perron eigenvector $X''$ of $T''$ we have $X''_{i-1}<X''_i$. We claim that $X''_i\leq X''_{i+1}$. If possible suppose that $X''_i>X''_{i+1}.$ Let $(m-1)\cdot\l_1(T')=\l.$
		 \begin{align*}
		 	&\l X''_i=\dfrac{m-1}{\l-m+2}X''_i+\dfrac{m-2}{\l-m+3}(X''_{i-1}+2X''_i+X''_{i+1})+X''_{i-1}+X''_{i+1}\\
		 	\implies&\left(\l-\dfrac{m-1}{\l-m+2}-\dfrac{2(m-2)}{\l-m+3}\right)X''_i=\dfrac{\l+1}{\l-m+3}(X''_{i-1}+X''_{i+1})<\dfrac{2(\l+1)}{\l-m+3}X''_i\\
		 	\implies&\l-\dfrac{m-1}{\l-m+2}-\dfrac{2(m-2)}{\l-m+3}<\dfrac{2(\l+1)}{\l-m+3}\\
		 	\implies&\l<\dfrac{1}{2}\left(m+\sqrt{m^2+12(m-1)}\right)<m+3.
		 \end{align*}
	 But here in $T''$ we have a hyperstar with $k-d+1$ edges as an induced sub hypergraph, so 
	 \begin{align*}
	 	(m-1)\l_1(T'')&>\dfrac{1}{2}\left(m-2+\sqrt{(m-2)^2+4(k-d+1)(m-1)}\right)\\
	 	&=\dfrac{1}{2}\left(m-2+\sqrt{(m+8)^2+4\left((k-d-4)(m-1)-20\right)}\right)\\
	 	&>m+3.
	 \end{align*}
 Therefore we have $X''_i<X''_{i+1}$. Now moving the pendant edge from $v_i$ to $v_{i+1}$, we have 
 $\l_1(T)<\l_1(T'')<\l_1(T_d(c_p=1,c_{p+1}=k-d-1))$. We also have $\l_1(T')<\text{max}\left\{\l_1(T_d(c_p=1,c_{p+1}=k-d-1)),\l_1(T_d(c_p=k-d))\right\}$\\
 Similarly, if $T'=T_d(c_{p+1}=k-d-l,c_i=l(i>p+1))$ then we have $\l_1(T')<\text{max}\left\{\l_1(T_d(c_{p+1}=k-d-1,c_{p+2}=1)),\l_1(T_d(c_p=k-d))\right\}$.\\
 
 \textbf{Subcase B :: $\mathbf{d(v_i)\leq 2}$ for all $\mathbf{i\neq p+1,p+2}$:} \\
 If $d(v_{p+1}),d(v_{p+2})>2$, then let $T''$ be the hypertree obtained by releasing all the non-pendant edges of $T$ other than the edges of $P_L(d)$. Now for the Perron eigenvector $X$ of $T'',$ we have either $X_{v_{p+1}}\geq X_{v_{p+2}}$ or $X_{v_{p+1}}<X_{v_{p+2}}$. If $X_{v_{p+1}}\geq X_{v_{p+2}}$, then move all but one pendant edges from $v_{p+2}$ to $v_{p+1}$ and if $X_{v_{p+1}}< X_{v_{p+2}}$, then move all but one pendant edges from $v_{p+1}$ to $v_{p+2}$. Thus in both the cases we have  $\l_1(T)<\l_1(T'')<\l_1(T_d(c_{p+1}=k-d-1,c_{p+2}=1))$.\\
 Next if only one of $d(v_{p+1}),d(v_{p+2})>2$, say $d(v_{p+1})>2,$ then there exist non-pendand edges containing the vertex $v_{p+1}$. Let $T'''$ be the hypertree obtained from $T$ by releasing all but one (say $e$ ) non-pendant edges containing $v_{p+1}$ (not in $P_L(d)$). Let $v(\neq v_{p+1})\in e$ be the other non-pendant vertex in $e.$ Now for the Perron eigenvector $X$ of $T'''$ we have either $X_{v_{p+2}}\geq X_v$ or $X_{v_{p+2}}< X_v$. If  $X_{v_{p+2}}\geq X_v$, then we move the edges (not $e$) containing $v$ from $v$ to $v_{p+2},$ and otherwise we move the edge $e_{p+3}$ from $v_{p+2}$ to $v$. Then release all the non-pendant edges other than the edges of $P_L(d).$ Then by the first part of this subcase we have the required reslut.\\
 
 \item \textbf{ Case II:: $\mathbf{d(v_{p+1})=d(v_{p+2})=2:}$}\\
 	Let $X_{v_t}=$max$\{X_{v_i}|i\neq p+1,p+2\},$ for the Perron eigenvector $X$ of $T.$ Now we move all the edges (not in $P_L(d)$) from $v_i's(i\neq p+1,p+2)$ to $v_t.$ Then let $T^4$ be the hypertree obtained by releasing all the non-pendant edges (not in $P_L(d)$) containing $v_t$. So we have $\l_1(T)<\l_1(T^4)$ and $T^4=T_d(c_i)$ for some $i\neq p+1,p+2.$ Since $\l_1(T_d(c_i=k-d))\leq \l_1(T_d(c_p=k-d))$ (for $i\neq p+1,p+2$) we have $\l_1(T)<\l_1(T^4)<\l_1(T_d(c_p=k-d))$.
 This completes the proof.
	\end{itemize}

\item Since $\l_1(T_d(c_{p+1}=k-d-1,c_{p+2}=1)=\l_1(T_d(c_p=1,c_{p+1}=k-d-1)$ and  for $k\geq \left((4d^2-1)(m-1)+2\right)/4$, we have $\l_1(T_d(c_i=k-d))\leq \l_1(T_d(c_p=k-d))$ for all $i\neq p+1.$.	
\end{enumerate}	
\end{proof}

\begin{lemma}\label{lemma6}
	Let $p=\left[\dfrac{d}{2}\right].$ Let $X$ and $Y$ be the Perron eigenvectors of $T_d(c_{p-1}=2,c_p=k-d-2)$ and $T_d(c_p=2,c_{p+1}=k-d-2),$ respectively. If $(k-d-6)(m-1)\geq 42$, then
	 $X_{v_p}\geq 2X_{v_{p-1}}$ and
		$Y_{v_p}\geq 2Y_{v_{p+1}}.$
\end{lemma}
\begin{proof}
	Let $\l=(m-1)\cdot\l_1(T_d(c_{p-1}=2,c_p=k-d-2))$. Then from the eigenvalue equation we have 
	\begin{align}
		&\left(\l-\dfrac{2(m-1)}{\l-m+2}-\dfrac{2(m-2)}{\l-m+3}\right)X_{v_{p-1}}=\dfrac{\l+1}{\l-m+3}(X_{v_p}+X_{v_{p-2}}) \notag \\
		\implies &\left(\l-\dfrac{2(m-1)}{\l-m+2}-\dfrac{\l+2m-3}{\l-m+3}\right)X_{v_{p-1}}<\dfrac{\l+1}{\l-m+3}X_{v_p}  \text{ (since $X_{v_{p-2}}<X_{v_{p-1}}$)} \notag \\
		\implies& AX_{v_{p-1}}<BX_{v_p} \ \ \text{(say).}
	\end{align}
	Now if $A<2B$, then 
	\begin{align*}
		&\l-\dfrac{2(m-1)}{\l-m+2}-\dfrac{\l+2m-3}{\l-m+3}<0\\
		\implies&\l<\left(m+1+\sqrt{\left(m(m+18)-11\right)}\right)/2<m+5.
	\end{align*}
Again we have $\l_1(T_d(c_{p-1}=2,c_p=k-d-2))>\l_1(T_2(c_2=k-d))$ and for  $(k-d-6)(m-1)\geq 42$,
	$\l_1(T_2(c_2=k-d))\geq (m+5)/(m-1)$. Similarly we have $Y_{v_p}\geq 2 Y_{v_{p+2}}$ and this completes the proof.
\end{proof}
\begin{theorem}
	Let  $T$ be an $m$-uniform hypertree with $k$-edges, and diameter $d$ such that $(k-d-6)(m-1)\geq 42$. If $T\neq T_d(C_{p+1}=k-d)$ where $p=\big[\dfrac{d}{2}\big]$ then 
	$\l_1(T)\leq \l_1(T_d(c_p=k-d))$ and the equality holds only when $T=T_d(c_p=k-d).$

\end{theorem}
\begin{proof}
	We have two cases
	\begin{enumerate}[$(i)$]
		\item Let $d=2p+1$. At first we release the edge $e_{p-1}$ of $T_1\in\{T_d(c_{p+1}=k-d-1,c_{p+2}=1),T_d(c_p=1,c_{p+1}=k-d-1)\}$. $T_2$ be the resultant hypertree. If $T_1=T_d(c_p=1,c_{p+1}=k-d-1),$ then $T_2=T_{d-1}(c_{p-1}=2,c_p=k-d-1).$ Also we note that $p=\left[\dfrac{d-1}{2}\right].$ Now for the Perron eigenvector $X$ of $T_2,$ using the above lemma, we have $X_p>2X_{p-1}.$ Let $f_1,f_2$ be the pendant edges incident to the vertex $v_{p-1},$ in $T_2$ and $v$ be a vertex adjacent to the vertex $v_{d-1}$ in $T_2.$ We take $E(v_{p-1})=\{f_1,f_2\}$ and $V_1=\{v,v_p\}.$ Then $T_2[E(v_{p-1});V_1]=T_d(c_p=k-d)$ and by theorem \ref{ES1} we have $\l_1(T_{d-1}(c_{p-1}=2,c_p=k-d-1))<\l_1(T_d(c_p=k-d)).$ \\
		 Next, let $T_1=T_d(c_{p+1}=k-d-1,c_{p+2}=1)$. Then $T_1^r(e_{p-1})=T_{d-1}(c_{p-1}=1,c_{p}=k-d-1,c_{p+1}=1)$. For the Perron eigenvector $Y$ of $T_3=T^r_1(e_{p-1}),$ we have $Y_{p-1}>Y_{p+1}$ or $Y_{p-1}\leq Y_{p+1}.$ If $Y_{p-1}>Y_{p+1}$
	   then we move the pendant edge from $v_{p+1}$ to $v_{p-1}$ otherwise we move the pendant edge from $v_{p-1}$ to $v_{p+1.}$ In the both cases, let $T_4$ be resultant hypertree. Then $T_4=T_{d-1}(c_{p-1}=2,c_{p}=k-d-1)$ or $T_4=T_{d-1}(c_p=k-d-1,c_{p+1}=2)$. Simlilarly, using the above lemma and theorem \ref{ES1}, we have $\l_1(T)<\l_1(T_4)<\l_1(T_d(c_p=k-d))$.
		
		\item Let $d=2p.$ We note that in this case $T_d(c_{p+1}=k-d-1,c_{p+2}=1)=T_d(c_p=1,c_{p+1}=k-d-1)$  and this completes the proof.
	\end{enumerate}
\end{proof}

\subsection{Hypertrees with largest to seventh-largest spectral radii}
\begin{theorem}Let $(k-10)(m-1)\geq 42.$ Then
$\l_1(T_2(c_2=k-2))>\l_1(T_3(c_2=k-3))>\l_1(T_3(c_2=k-4,c_3=1))>\l_1(T_4(c_3=k-4))>\l_1(T_4(c_2=k-4))>\l_1(T_3(c_2=k-5,c_3=2))>\l_1(T_4(c_2=1,c_3=k-5)),$ and these are the hypertrees with first seventh largest spectral radius.
\end{theorem}
\begin{proof}
	We note that  $T_3(c_2=k-3),T_3(c_2=k-4,c_3=1),T_3(c_2=k-5,c_3=2),T_3(c_2=k-6,c_3=3)$ are the hypertrees with first four largest spectral radii among the hypertrees with diameter three. Again
	$T_4(c_3=k-4),T_4(c_2=k-4),T_4(c_3=k-5,c_4=1)$ are the hypertrees with first three largest spectral radii among the hypertrees with diameter four. Now
	\begin{itemize}
		\item Let $X$ be the Perron eigenvector of $T_4(c_3=k-4).$ Then we have $X_2=X_4$. Move the pendant edge from the vertex 2 to the vertex 4. Here also we get $T_3(c_2=k-4,c_3=1)$ as the resultant hypergraph. So we have $\l_1(T_4(c_3=k-4))<\l_1(T_3(c_2=k-4,c_3=1)).$ 
		\item Using the technique used in lemma \ref{lemma6} we have $\l_1(T_3(c_2=k-5,c_3=2))<\l_1(T_4(c_2=k-4))$ and $\l_1(T_3(c_2=k-6,c_3=3))<\l_1(T_4(c_2=1,c_3=k-4)).$ 
		\item Again we have $\l_1(T_5(c_3=k-5))<\l_1(T_4(c_2=1,c_3=k-4))<\l_1(T_3(c_2=k-5,c_3=2)).$ 
	\end{itemize}

\end{proof}
 
 We rewrite the theorem $4.2.2$ from \cite{ASAB2020}, as follows
 \begin{theorem}
 	The adjacency eigenvalues of an loose cycle $C_L(k)$ of length $k$, are
 	\begin{enumerate}[(i)]
 		\item \label{case2} 
 		$-1/(m-1)$ with the multiplicity {at least} $k(m-3),$ and
 		\item 
 		$\gamma^{+}_i/(m-1)$,~$\gamma^{-}_i/(m-1)$ with the multiplicity {at least}  one, where, $$\gamma_i^{\pm}=\dfrac{1}{2}\Bigg[m-3+2\cos{\dfrac{2\pi{i}}{k}} \pm\sqrt{(m-3+2\cos{\dfrac{2\pi{i}}{k}})^2+8(m-2+\cos{\dfrac{2\pi{i}}{k}})}\Bigg],$$ for $i=1,2,\dots,k$, when $m\geq{3}.$
 	\end{enumerate}
 \end{theorem}
 From this theorem we have the spectral radius of $C_L(k)$ is $\left(m-1+\sqrt{m^2+6m-7}\right)/(2m-2).$

\begin{proposition}
	Let $P$ be an $m$-uniform loose path. Then the spectral radius of $A_{P}$, $\l_1(P)<\left(m-1+\sqrt{m^2+6m-7}\right)/(2m-2).$
\end{proposition}

\begin{proof}
	From theorem 4.2.2 of \cite{ASAB2020}, we have spectral radius of a loose cycle is $\left(m-1+\sqrt{m^2+6m-7}\right)/(2m-2)$. Again we can think loose path is an induced hypergraph of the loose cycle and this completes the proof.
\end{proof}

\section{ Linear unicyclic hypergraphs with largest, second-largest, and third-largest spectral radii in $\mathcal{P(\H)}$ for any given length of the cycle}

\subsection{ Largest spectral radius with fixed cycle length}

Let $UC_l(c_1,c_2,\dots,c_l)$ be the linear hypergraph obtained by attaching $c_1,c_2,\dots,c_l$ number pendant edges, to the core vertices $v_1,v_2,\dots,v_l$ respectively, of the loose cycle $C_L=v_1e_1v_2e_2\dots v_le_lv_1$ of length $l.$ We write $UC_l(c_{i_1},c_{i_2},\dots,c_{i_r})$ to denote the hypergraph $UC_l(c_1,c_2,\dots,c_l),$ when $c_i=0,$ for $i\neq i_1,i_2,\dots,i_r.$
\begin{lemma}\label{UC0}
	Let $l\geq 4$. Then $\l_1(UC_l(c_1=p))<\l_1(UC_{l-1}(c_1=p+1))$.
\end{lemma}
\begin{proof}
	Let $e$ be an edge of $UC_l(c_1=p)$ adjacent to the pendant edges. Now we release the edge $e$ and get $UC_{l-1}(c_1=p+1)$ as the resultant hypergraph. Hence $\l_1(UC_l(c_1=p))<\l_1(UC_{l-1}(c_1=p+1)).$
\end{proof}

Let $\mathcal{UC}(l;k)$ be the collection of the unicyclic hypergraph with $k$ edges and $l$ length cycle.
\begin{theorem}\label{UCT1}
	Let $UC^*_l\in \mathcal{UC}(l;k) $ be such that $\l_1(UC^*_l)=max \left\{\l_1(\H)| \H \in \mathcal{UC}(l;k) \right\}$. Then
	\begin{enumerate}[$(i)$]
		\item $UC^*_i=UC_i(c_1=k-i),$ \text{ and $\l_1(UC_i)< \l_1(UC_i(c_1=k-i))$ for other $UC_i\in \mathcal{UC}(i;k).$}
		\item  $\left(m-1+\sqrt{m^2+6m-7}\right)/(2m-2)\leq \l_1(UC^*_{l})<\l_1(UC^*_{l-1})$ for all $l,$ and the equality holds only when $UC^*_{l}=C_L(k).$ 
		\item $\dfrac{m-2+\sqrt{(m-2)^2+4(k-l+2)(m-1)}}{2(m-1)}<\l_1(UC_l(c_1=k-l))<\dfrac{m+\sqrt{m^2+4(k-l+2)(m-1)}}{2(m-1)}.$
	\end{enumerate}
\end{theorem}

\begin{proof}
	\begin{enumerate}[$(i)$]
		\item Let $C_L=u_1f_1u_2f_2\dots u_lf_lu_1$ be the cycle in $\H$ of length $l$ and $\H\neq UC_l(c_1=k-l)$.
		Let $X$ the Perron eigenvector for the hypergraph $\H$ and $X_{u_q}=\text{max}\{X_{u_i}|i=1,2,\dots,p\}$. Let $\H_1$ be the hypergraph obtained by moving the edges (not in $C_L$) from $u_1,u_2,\dots,u_p$ to $u_q.$ Then we have $\l_1(\H_1)>\l_1(\H)$. Next, let $\H_2$ be the hypergraph obtained by releasing all the non-pendant edges of ${\H}_1$ (other than the edges of the cycle) at $u_q$. Then we have $\l_1(\H_2)>\l_1(\H_1)>\l_1(\H)$ and  $\H_2=UC_l(c_1=k-l)$.
		\item First we note that every unicyclic hypergraph $\H$ contains, $C_L(l)$ as a subhypergraph and $\l_1(C_L(l))=\left(m-1+\sqrt{m^2+6m-7}\right)/(2m-2)$.
		
		 Now let $d(v_1)=k-l+2$ and $e_1$ be an edge in the cycle of $UC_l(c_1=k-l)$ with $v_1\in e_1.$
		Then we have $UC_{l-1}(c_1=k-l+1)=UC_l(c_1=k-l)^r(e_1)$ and this completes the proof.

		\item  
		Since $T_2(c_2=k-l+2)$, the hyperstar with $(k-l+2)(m-1)+1$ vertices, is a subhypergraph of $UC_l(c_1=k-l)$ thus by the crollary~\ref{hyperstar} we have $$\l_1(UC_l(c_1=k-l))>\left(m-2+\sqrt{(m-2)^2+4(k-l+2)(m-1)}\right)/(2m-2).$$
		
		Now let $C_L=v_1e_1v_2e_2v_3\dots v_le_lv_1$ be the cycle in $UC_l(c_1=k-l)$, and $f_1,f_2,\dots,f_{k-l}$ be the pendant edges of $UC_l(c_1=k-l)$ containing the vertex $v_1.$
		Let $(X,\l_1)$ be the Perron eigen-pair of $A_{UC_l(c_1=k-l)}$ and $\l=(m-1)\cdot\l_1$. Considering the eigenvalue equation for any pendant vertex $v\sim v_1$ (not in $C_L$), we have 
		\begin{align*}
		X_v&=X_{v_1}/(\l-m+2).
		\end{align*}
		Also for any $u\in V''_{e_1}$,
		\begin{align*}
			 X_u&=(X_{v_1}+X_{v_2})/(\l-m+3).
		\end{align*}
	and for any $u'\in V''_{e_l},$ 
	\begin{align*}
		X_{u'}&=(X_{v_1}+X_{v_l})/(\l-m+3).
	\end{align*}
		Using these we have
		\begin{align*}
		 \left(\l-\dfrac{(k-3)(m-1)}{\l-m+2}-\dfrac{2(m-2)}{\l-m+3}\right)X_{v_1}=\dfrac{(\l+1)(X_{v_1}+X_{v_l})}{\l-m+3}
		\end{align*}
		Again we have $X_{v_1}>X_{v_2}$, otherwise let $\H_1$ be the hypergraph obtained from $UC_l(c_1=k-l)$ by moving the pendant edges from $v_1$ to $v_2$. Then $\l_1(UC_l(c_1=k-l))<\l_1(\H_1)$ and which is not possible because $\H_1=UC_3(c_1=k-3)$. Thus we have $X_{v_1}>X_{v_2},X_{v_l}$ and so from the above equality, we get
		\begin{align*}
			&\l-\dfrac{(k-3)(m-1)}{\l-m+2}-\dfrac{2(m-2)}{\l-m+3}<\dfrac{2(\l+1)}{\l-m+3}\\
			\implies &\l-\dfrac{(k-3)(m-1)}{\l-m+2}-\dfrac{2(m-2)}{\l-m+2}<\dfrac{2(\l+1)}{\l-m+3}<\dfrac{2(\l+1)}{\l-m+2} \\
			\implies & \l <\left(m+\sqrt{m^2+4(k-1)(m-1)}\right)/2.
		\end{align*}
		So we have $\l_1(UC_l(c_1=k-l)) <\left(m+\sqrt{m^2+4(k-l+2)(m-1)}\right)/(2m-2),$ and this completes the proof.
	\end{enumerate}
\end{proof}

\begin{note}
	For any unicyclic hypergraph $\H$ with $k$ edges we have $\l_1(\H)\leq \l_1(UC_2(c_1=k-2))$ and the equality holds only when $\H=UC_2(c_1=k-2)$.
\end{note}

\subsection{ Second-largest spectral radius with fixed cycle length}

\begin{lemma}\label{UCL7}\sloppy
	Let $p=\Big[\dfrac{l+1}{2} \Big]$ and $a\leq b$ be such that $\l_1(UC_l(c_1=b,c_p=a))\geq \left(m+\sqrt{m^2+4(a+2)(m-1)}\right)/(2m-2).$ Then  $\l_1(UC_l(c_1=b,c_i=a))<\l_1(UC_l(c_1=b,c_{i-1}=a))$ for $i=2,3,\dots,p.$	
\end{lemma}

\begin{proof}
	Let $X$ be the Perron eigenvector of $UC_l(c_1=b,c_p=a)$ and $\l=(m-1)\cdot\l_1(UC_l(c_1=b,c_p=a))$. Then from the eigenvalue equation we have
	\begin{align*}
		\left(\l-\dfrac{a(m-1)}{\l-m+2}-\dfrac{2(m-2)}{\l-m+3}\right)X_{v_p}=\dfrac{\l+1}{\l-m+3}\left(X_{v_{p-1}}+X_{v_{p+1}}\right)
	\end{align*}
	Now if $l$ is even, then $X_{v_{p-1}}=X_{v_{p+1}}$ and if $l$ is odd then $X_{v_{p+1}}<X_{v_p}.$
	So if $X_{v_p}<X_{v_{p-1}}$ then we have 
	\begin{align*}
		&\left(\l-\dfrac{a(m-1)}{\l-m+2}-\dfrac{2(m-2)}{\l-m+3}\right)X_{v_p}<\dfrac{2(\l+1)}{\l-m+3}X_{v_p}\\
		\implies&\left(\l-\dfrac{a(m-1)}{\l-m+2}-\dfrac{2(m-2)}{\l-m+3}\right)<\dfrac{2(\l+1)}{\l-m+3}\\
		\implies&\l-\dfrac{a(m-1)}{\l-m+2}-\dfrac{2(\l+m-1)}{\l-m+2}<0\\
		\implies&\l<\left(m+\sqrt{m^2+4(a+2)(m-1)}\right)/2.
	\end{align*}
	Since here $\l_1\geq \left(m+\sqrt{m^2+4(a+2)(m-1)}\right)/(2m-2)$, we have $X_{v_{p-1}}>X_{v_p}$.  Therefore $\l_1(UC_l(c_1=b,c_p=a))<\l_1(UC_l(c_1=b,c_{p-1}=a))$ and	 $\l_1(UC_l(c_1=b,c_i=a))<\l_1(UC_l(c_1=b,c_{i-1}=a))$ for $i=2,3,\dots,p.$	 
\end{proof}

\begin{theorem}\label{UCT2}
	Let $\H\in{\PH}$ be an unicyclic hypergraph with $k$-edges and the length of the cycle be $l$. Also let $(k-l-6)(m-1)\geq 20.$ 
	\begin{enumerate} [$(i)$]
		\item If $\H\neq UC_l(c_1=k-l)$ then $\l_1(\H)\leq \l_1(UC_l(c_1=k-l-1,c_2=1))$, and the equlaity holds only when $\H=UC_l(c_1=k-l-1,c_2=1)$.
		\item $\l_1(UC_l(c_1=k-l-1,c_2=1))<\left(m+\sqrt{m^2+4(k-l+1)(m-1)}\right)/(2m-2).$
	\end{enumerate} 	
\end{theorem}
\begin{proof}
	\begin{enumerate}[(i)]
	\item Since $\H\neq UC_l(c_1=k-l)$ we can choose $v_1$ such that $d(v_1)>2$.
	For the Perron eigenvector $X$ of $\H$, let $X_{v_t}=$max$\{X_{v_i}|i\neq 1\}$. Let $\H_1$ be the hypergraph obtained by moving the edges (not in the cycle), from $v_i's$ to $v_t$ and then releasing all the non-pendant edges (not in $C_L$) at $v_t$ repeatedly. Here we have two cases
	\begin{enumerate}[$(a)$]
		\item \textbf{Case I :: All the edges (not in the cycle) incident to $\mathbf{v_1}$ are pendant:} Then $\H_1=UC_l(c_1=d(v_1)-2,c_i=a)$ for some $a$ and $i.$ Again we have $\l_1\left(UC_l(c_1=d(v_1)-2,c_i=a)\right)<\l_1\left(UC_l(c_1=k-l-1,c_i=1)\right)$. If $i=2$, then we are done. Let $i>2.$ We note that $UC_l(c_1=k-l-1,c_i=1),$ contains the hyperstar $T_2(c_2=k-l-1)$ as a subh ypergraph. So 
		\begin{align*}
			\l_1(UC_l(c_1=k-l-1,c_i=1))&>\left(m-2+\sqrt{(m-2)^2+4(k-l+1)(m-1)}\right)/(2m-2)\\
			&=(m+3)/(2m-2) \ \ \ \ \text{since $(k-l-4)(m-1)\geq 20$}\\
			&>(m+\sqrt{m^2+12(m-1)})/(2m-2)
		\end{align*}
		Therefore using lemma \ref {UCL7} we have 
		  $\l_1\left(UC_l(c_1=k-l-1,c_i=1)\right)<\l_1\left(UC_l(c_1=k-l-1,c_2=1)\right).$
		\item \textbf{Case II::There is a non-pendant edge (not in the cycle) incident to $\mathbf{v_1}$, say $\mathbf{e}$:} Let $v(\neq v_1)$ be the other non-pendant vertex in $e$. Let $X'$ be the Perron eigenvector of $\H_1$. If $X'_v\geq X_{v_2}$ then move all (but $e_1$) the edges from $v_t$ to $v$ and if $X'_v< X_{v_2}$, then move all the edges (not e), containing $v$, from $v$ to $v_t$. So we can consider $v$ or $v_2$ as the vertex $v_2$ in the cycle. Then release all the non-pendant edges (not in the cycle) and let $X''$ be the Perron eigenvector of the resultant hypergraph $\H_2.$ Now if $X''_{v_2}\geq X''_{v_t}$ then move the pendant edges from $v_t$ to $v_2$ otherwise move the pendant edges from $v_2$ to $v_t$ and for some $a,b,i$ we get $UC_l(c_1=b,c_i=a)$ as the resultant hypergraph. Hence $\l_1(\H)<\l_1(\H_1)<\l_1(\H_2)<\l_1\left(UC_l(c_1=b,c_i=a)\right)<\l_1\left(UC_l(c_1=k-l-1,c_i=1)\right)<\l_1\left(UC_l(c_1=k-l-1,c_2=1)\right).$
	\end{enumerate}
\item Similar to the proof of $(iii)$ of theorem \ref{UCT1}.
\end{enumerate}	
\end{proof}

\subsection{ Third-largest spectral radius with fixed cycle length}

\begin{lemma} \label{UCL1}
	Let $l_1\geq l_2\geq 1$ and $l\geq 3$. Then
	$\l_1(UC_l(c_i=l_1,c_j=l_2))<\l_1(UC_l(c_i=l_1+1,c_j=l_2-1))$ for any $i,j=1,2,\dots,l.$
\end{lemma}
\begin{proof}
	Let $f_1,f_2,\dots,f_{l_1}$ and $g_1,g_2,\dots,g_{l_2}$ be the pendant edges of $UC_l(c_i=l_1,c_j=l_2)$ containing $v_i$ and $v_j$, respectively.
	Let $X$ be the Perron eigenvector of $UC_l(c_i=l_1,c_j=l_2)$. Then we have $X_{v_i}\geq X_{v_j}$ or $X_{v_j}\geq X_{v_i}$. If $X_{v_i}\geq X_{v_j}$, then we move one pendant edge from $v_j$ to $v_i$. On the other hand, if $X_{v_j}\geq X_{v_i}$, then we move $l_1-l_2+1$ pendant edges from $v_j$ to $v_i$. In both the cases the resultant hypergraph is, $UC_l(c_i=l_1+1,c_j=l_2-1)$. This completes the proof.
\end{proof}

Let $U_lC(c_1=k-l-1)$ be denotes the hypergraph obtained by attaching an edge, to an pendant edge of $UC_l(c_1=k-l-1)$. Note that number of edges in $U_lC(c_1=k-l-1)$ is $k-l-1+1=k.$

\begin{lemma}
	Let $\H\in{\PH}$ be an unicyclic hypergraph with $k$ edges. If $\H\notin \left\{UC_l(c_1=k-l),UC_l(c_1=k-l-1,c_2=1)\right\}$ then $\l_1(\H)\leq \text{max}\{\l_1(UC_l(c_1=k-l-2,c_2=2)),\l_1(UC_l(c_1=k-l-2,c_3=1)),\l_1(U_lC(c_1=k-l-1))\}.$
\end{lemma}
\begin{proof}
	Let $C_L=v_1e_1v_2e_2v_3e_3\dots v_le_lv_1$ be the cycle in $\H.$ Also let $L=\{v\in\{v_1,v_2,\dots,v_l\}|d(v)>2\}.$ We have the following cases,
	\begin{enumerate}[(i)]
		\item \textbf{Case I:: $\mathbf{|L|=1}$ and $\mathbf{v_1\in L}$.}
		We have two following subcases,
		\begin{itemize}
			\item \textbf{Subcase I : $\mathbf{d(v_1)=3.}$ } 
			Since $\H\neq UC_l(c_1=k-l)$, there exists a non-pendant edge, say $e$ (not in $C_L$ ) containing $v_1.$ Let $\H_1$ be the hypergraph obtained from $\H$ by releasing all the non-pendant edges of $\H$ other than $e$ and the edges in $C_L$. Let $u\in e$ be the other pendant vertex in $e$ and $v\in e$ be a pendant vertex in $\H_1.$ Let $X$ be the Perron eigenvector of $\H_1$. If $X_{v_1}\geq X_u$, then by moving $k-l-2$ pendant edges from $u$ to $v_1$, we get $U_lC(c_1=k-l-2)$ as the resulting hypergraph. Thus we have $\l_1(\H)<\l_1(U_lC(c_1=k-l-2)).$ Now we suppose $X_{v_1}<X_u$. Let $v$ be a vertex adjacent to $u$, but not to$v_1.$ Here, $X_{v_1}>X_{v_2}=X_{v_l}$, so from the eigenvalue equation we have  
			$$
			\left(\l-\dfrac{2(m-2)}{\l-m+3}\right)X_{v_2}<\dfrac{2(\l+1)}{\l-m+3}X_{v_1}
			$$ and 
			\begin{align*}
				\left(\l-\dfrac{3(m-2)}{\l-m+3}\right)X_{v_1}=\dfrac{\l+1}{\l-m+3}(X_u+2X_{v_2})
			\end{align*}
			Combining the above two we have
			\begin{align*}
				\left(\left(\l-\dfrac{2(m-2)}{\l-m+3}\right)\left(\l-\dfrac{3(m-2)}{\l-m+3}\right)-\dfrac{4(\l+1)^2}{\l-m+3}\right)X_{v_2}&<\dfrac{2(\l+1)^2}{\l-m+3}X_u\\
				&=\dfrac{2(\l-m+2)(\l+1)^2}{\l-m+3}X_v.
			\end{align*}
			Here 
			\begin{align*}
				\left(\left(\l-\dfrac{2(m-2)}{\l-m+3}\right)\left(\l-\dfrac{3(m-2)}{\l-m+3}\right)-\dfrac{4(\l+1)^2}{\l-m+3}\right)-\dfrac{2(\l-m+2)(\l+1)^2}{\l-m+3}\\
				=\dfrac{(\l-m+3)\left(\l.\{\l(\l-m+1)-5m+4\}-2m\right)+4(m-2)^2}{(\l-m+3)^2}	
			\end{align*}
			Again for $(k-l-6)(m-1)\geq 30,$ we have $\l_1(\H_1)\geq (m+4)/(m-1),$ and so we have 
			\begin{align*}
				(\l-m+3)\left(\l.\left(\l(\l-m+1)-5m+4\right)-2m\right)+4(m-2)^2>0.	
			\end{align*} which implies that $X_v>X_{v_2}.$ Now let $\H_3$ be the hypergraph obtained from $\H_1$ by moving the edge $e_2$ from $v_2$ to $v$ and then releasing the edge $e_l.$ Then we have $\l_1(\H_1)<\l_1(\H_3)$ and $\H_3=UC_l(c_1=k-l-2,c_2=2).$
			\item \textbf{Subcase II : $\mathbf{d(v_1)>3.}$ }Let $\H_4$ be the hypergraph obtained from $\H$ by releasing all the non-pendant edges not containing the vertices of $C_L.$ If $\H_4=U_lC(c_1=k-l-1)$ then we are done. If not, then let  $W=\{w |w\sim v_1, \text{ and $w$ is non-pendant in ${\H}_1$}\}$ and $X_v=\text{max}\{X_w| w\in W\}$. Let $\H_2$ be the hypergraph obtained from $\H_1$ by moving all the edges (not containing $v_1$) from the vertices $w\in{W}$ to $v.$
			Then $\H_2=UC_l(c_1,c_2)$ for some $c_1,c_2\geq 2$ and so by lemma \ref{UCL1} we have
			$\l_1(\H_1)<\l_1(\H_2)\leq \l_1(UC_l(c_1=k-5,c_2=2))$.
		\end{itemize} 
		\item \textbf{Case II :: $\mathbf{|L|>1:}$} Here also we have two subcases
		\begin{itemize}
			\item \textbf{Subcase I :Let $\mathbf{\H}$ has no non-pendant edge other than the edges in $\mathbf{C_L}.$}
			Then we have  $\l_1(\H)\leq \l_1(UC_l(c_1=k-l-1,c_3=1))$ or $\l_1(\H)\leq \l_1(UC_l(c_1=k-l-2,c_i=2))$ for some $i=2,3,\dots,l.$ 
			Again for $(k-l-6)(m-1)\geq 30,$ we have $$\l_1(UC_l(c_1=k-l-2,c_i=2))\geq \dfrac{m+\sqrt{m^2+16(m-1)}}{2(m-1)}$$ and so by using the lemma \ref{UCL7} we have $\l_1(\H)\leq \l_1(UC_l(c_1=k-l-2,c_2=2)).$
			\item \textbf{Subcase II : Let $\mathbf{\H}$ has non-pendant edges other than the edges in $\mathbf{C_L.}$} For the Perron eigenvector $X$ of $\H$ let $X_{v_t}=\text{max}\{X_v:v\in L\}.$ Let $\H_5$ be the hypergraph obtained from $\H$ by moving the edges (not in $C_L$) from the vertices $v\in L$ to $v_t.$ We have $\l_1(\H)<\l_1(\H_5)$ and then by Subcase II of case I, we have the required result. 
		\end{itemize}
		
	\end{enumerate}	
\end{proof}

\begin{lemma}
	Let $(k-l-5)(m-1)\geq 12.$ Then 
 \begin{enumerate}[$(i)$]
		\item $\l_1(UC_l(c_1=k-l-1,c_2=2))<\text{min}\{\l_1(UC_l(c_1=k-l-1,c_3=1)),\l_1(U_lC(c_1=k-l-1))\},$ and
		\item $\l_1(UC_l(c_1=k-l-1,c_3=1))<\l_1(U_lC(c_1=k-l-1))$ for $l\geq 4.$
 \end{enumerate}
\end{lemma}
\begin{proof}
	\begin{enumerate}[$(i)$]
		\item
	Let $\l=(m-1)\l_1.$ Then from the eigenvalue equation we have,
	\begin{align*}\left(\l-\dfrac{2(m-1)}{\l-m+2}-\dfrac{2(m-2)}{\l-m+3}\right)X_{v_2}&=\dfrac{\l+1}{(\l-m+3)}(X_{v_1}+X_{v_3})\\
		&<\dfrac{2(\l+1)}{(\l-m+3)}X_{v_1}\\
		\implies AX_{v_2}&<BX_{v_1} \ \ \ \ \text{ (Say)}
	\end{align*}
	Then 
	\begin{align*}
		A-2B&=\left(\l-\dfrac{2(m-1)}{\l-m+2}-\dfrac{2(m-2)}{\l-m+3}\right)-\dfrac{4(\l+1)}{(\l-m+3)}\\
		&>\dfrac{\l(\l-m+2)-4m+2}{\l-m+2}>0,\ \ \ \ \text{ since $\l>m+2.$} 
	\end{align*}
	Therefore $X_{v_3}>2X_{v_2}$ and so by using theorem \ref{ES1} we have $$\l_1(UC_l(c_1=k-l-1,c_2=2))<\l_1(U_lC(c_1=k-l-1))$$ and $$ \l_1(UC_l(c_1=k-l-1,c_2=2))<\l_1(UC_l(c_1=k-l-1,c_3=1)).$$
	
	\item  We note that $l>3.$ Let $u$ be a pendant vertex (not in the cycle) adjacent to the vertex $v_3$  in $UC_l(c_1=k-l-1,c_3=1)$. Then we have 
	\begin{align*}
		\left(\left(\l-\dfrac{2(m-2)}{\l-m+3}-\dfrac{m-1}{\l-m+2}\right)\left(\l-\dfrac{2(m-2)}{\l-m+3}\right)-\dfrac{2(\l+1)^2}{(\l-m+3)^2}\right)X_{v_3}<\dfrac{2(\l-m+2)(\l+1)^2}{(\l-m+3)^2}X_{u}.
	\end{align*}
	Now 
	\begin{align*}
		&\left(\left(\l-\dfrac{2(m-2)}{\l-m+3}-\dfrac{m-1}{\l-m+2}\right)\left(\l-\dfrac{2(m-2)}{\l-m+3}\right)-\dfrac{2(\l+1)^2}{(\l-m+3)^2}\right)\\
		=&\dfrac{(\l-m+3)^2\{\l^3-m\l^2-5(m-1)\l\}}{(\l-m+2)(\l-m+3)^2}\\
		\ \ \ \ &+\dfrac{(\l-m+3)\{2\l^2+4(m-1)\l+2(m-1)(m-2)+2\}+4(m-2)^2(\l-m+2)}{(\l-m+2)(\l-m+3)^2}\\
		&>0,\ \ \ \ \text{ since $\l>m+1.$} .
	\end{align*}
	
	Thus we have $X_{v}>X_{v_3},$ and hence $\l_1(UC_l(c_1=k-l-1,c_3=1))<\l_1(U_lC(c_1=k-l-1)).$
	\end{enumerate}
\end{proof}
Now from these two lemmas we have the unicyclic hypergraph with a fixed cycle length having third largest spectral radius.
\begin{theorem}\label{UCT3}
	Let $\H\in \PH$ be a linear unicyclic hypergraph with $k$-edges and $\H \neq UC_l(c_1=k-l-1,c_2=1), UC_l(c_1=k-l)$, where $(k-l-5)(m-1)\geq 12.$ Then $\l_1(\H) \leq \l_1(U_lC(c_1=k-l-1))$, and the equality holds only when $\H=U_lC(k-l-1).$
\end{theorem}

\subsection{ Unicyclic hypergraphs with largest to third largest spectral radii in ${\PH}$}

We have first three linear unicyclic hypergraphs with length of the cycle $3$, having largest, second largest, and third largest spectral radius $UC_3(c_1=k-3),UC_3(c_1=k-4,c_2=1),$ and $U_3C(c_1=k-4),$ respectively. From theorem \ref{UCT1}, it is clear that, $UC_3(c_1=k-3)$ posseses the largest spectral radius among all the unicyclic hypergraphs in $\PH.$ Now our claim is $UC_3(c_1=k-k,c_2=1)$ and $U_3C(c_1=k-4)$ have the second and third largest spectral radius, respectively for the same.
To show this we have the following lemma.

\begin{lemma}
	Let $(k-7)(m-1)\geq 20$. Then $\l_1(UC_4(c_1=k-4))<\l_1(U_3C(c_1=k-4)).$
\end{lemma}
\begin{proof} Let $v$ be a pendant vertex in $UC_4(k-4)$ adjacent to the vertex $v_1.$ Then we have 
	\begin{align*}
		(\l-m+2)X_v=X_{v_1},
	\end{align*}
	and
	\begin{align*}
		\left(\left(\l-\dfrac{2(m-2)}{\l-m+3}\right)^2-\dfrac{2(\l+1)^2}{(\l-m+3)^2}\right)X_{v_3}=\dfrac{2(\l+1)^2}{(\l-m+3)^2}X_{v_1}
	\end{align*}
	Now 
	\begin{align*}
		&\left(\left(\l-\dfrac{2(m-2)}{\l-m+3}\right)^2-\dfrac{2(\l+1)^2}{(\l-m+3)^2}\right)-(\l-m+2)\dfrac{2(\l+1)^2}{(\l-m+3)^2}\\
		=&\dfrac{(\l-m+3)\left(\l\left(\l^2-(m-1)\l-4(m+1)\right)+6\right)+4(m-2)^2}{(\l-m+3)^2}\\
		&>0, \ \ \text{ since $\l>m+2.$}
	\end{align*}
	Thus we have $X_{v}>X_{v_3}$ and hence $\l_1(UC_4(c_1=k-4))<\l_1(U_3(3;k)).$	
\end{proof}

Using the theorems \ref{UCT1},\ref{UCT2},\ref{UCT3}, and the above lemma we have the following theorem,

\begin{theorem}
	Let $(k-7)(m-1)\geq 20.$ Then $UC_3(c_1=k-3),UC_3(c_1=k-4,c_2=1),$ and $U_3C(c_1=k-4)$ are the unicyclic hypergraphs in $\PH$ having first three largest spectral radius.
\end{theorem}

Now we estimate the largest spectral radius of unicyclic hypergraphs in $\PH.$
\begin{theorem}\label{UCT2}
	Let $\H$ be any linear unicyclic hypergraph with $k$ edges and $\alpha$ be the root of the polynomial $\left(x^2-(m-2)x-2m+3\right)\left(x^3-(m-3)x^2+(k-5)(m-1)x+2x-(k-3)(m-1)(m-3)\right)-2x(x+1)^2$ with the largest modulus. Then
	\begin{enumerate}[$(i)$]
		\item
		$\l_1(\H)\leq  |\alpha|/(m-1)$ and the equality holds only if $\H=UC_3(c_1=k-3).$
		\item $\left(m-2+\sqrt{(m-2)^2+4(k-1)(m-1)}\right)/2<|\alpha|<\left(m+\sqrt{m^2+4(k-1)(m-1)}\right)/2.$\label{UC2}
	\end{enumerate}
\end{theorem}

\begin{proof}
	
	\begin{enumerate}[(i)]
		\item
		The quotient matrix $B$ corresponding to the equitable partition $\pi_p$ for the hypergraph $UC_l(c_1=k-3)$ is given by 
		$$
		B=\dfrac{1}{m-1}
		\begin{bmatrix}
			0&2&(k-3)(m-1)&2(m-2)&0\\
			1&1&0&m-2&m-2\\
			1&0&0&0&0\\
			1&1&0&m-3&0\\
			0&2&0&0&m-3	
		\end{bmatrix}
		$$
		Using lemma \ref{lemma0} we have the
		characteristic polynomial of the matrix of $B'=(m-1)B$,
		
		\begin{align*}
			f_{B'}(x)&=\text{det}(B'-xI_5)\\
			&=\text{det}((m-3-x)I_2)*\text{det}\left(\left(A_{11}-xI_3-A_{12}(m-3-x)^{-1}I_3A_{21}\right)\right)
		\end{align*}
		where $A_{11}=
		\begin{bmatrix}
			0&2&(k-3)(m-1)\\
			1&1&0\\
			1&0&m-2
		\end{bmatrix}$,~
		$
		A_{12}=
		\begin{bmatrix}
			2(m-2)&0\\
			m-2&m-2\\
			0&0
		\end{bmatrix}
		$,
		~
		$A_{21}=
		\begin{bmatrix}
			1&1&0\\
			0&2&0
		\end{bmatrix}
		.$
	So,	
		\begin{align*}
			f_{B'}(x)&=\dfrac{1}{(m-3-x)}\text{det}\{(m-3-x)A_{11}-x(m-3-x)I_3-A_{12}A_{21}\}\\
			\implies f_{B'}(x)&=\left(x-1)(x-m+3)-3(m-2)\right)[(x-m+3)\left(x(x-m+2)-(k-1)(m-1)+2 \right)\\ & \ \ \ \ +2(m-2)]-2x(x+1)^2(x-m+2).
		\end{align*}
		\item Follows from theorem \ref{UCT1}
	\end{enumerate}
	
\end{proof}

\section{Linear bicyclic hypergraphs with largest and second-largest spectral radii in $\PH$}

Let $BC(l_1)$ be the linear hypergraph obtained by joining two loose cycles of length $3$ via identifying their one core vertex each and attaching $l_1$ numbers of pendant edges at the new core vertex. Let $\mathcal{BC}$ be the set of bicyclic hypergraphs in $\PH.$

\begin{theorem}\label{Bicyclic}
	Let $\H\in \PH$ be a bicyclic linear hypergraph with $k$-edges. Then 
	\begin{enumerate}[$(a)$]
		\item $\l_1(\H)\leq \l_1(BC(k-6))$ and the equality holds only when $\H=BC(k-6)$.
		\item $\l_1(BC(k-6))= |\beta|,$ where
		$\beta$ is the root of the polynomial  $\left(\l(\l-m+2)(\l-m+3)-(k-6)(m-1)(\l-m+3)- 4(m-2)(\l-m+2)\right)\left(\l^2-(m-2)\l-2m+3\right)-4(\l+1)^2(\l-m+2)=0$  with the largest modulus and $$\left(m-2+\sqrt{(m-2)^2+4(k-2)(m-1)}\right)/2<|\beta|<\left(m+2+\sqrt{(m+2)^2+4(k-2)(m-1)}\right)/2.$$
	\end{enumerate} 
\end{theorem}
\begin{proof}
	\begin{enumerate}[$(a)$]
		\item 
	Let $\H \neq BC(k-6)$ and $C_{L_1},C_{L_2}$ be two cycles in $\H$. Now we have the following two cases.
	\begin{itemize}
		\item \textbf{Case 1 : $\mathbf{C_{L_1}}$ and $\mathbf{C_{L_2}}$ have a common vertex, say $\mathbf{v}$.} We release all the non-pendant edges which are not in the cycles. If length of any cycle is greater than three, then release non-pendant edges of that cycle one by one, untill the length becomes three. So we get a hypergraph, say $\H_1$ consisting of two loose cycles of length $3$ attached at one core vertex and some pendant edges. Also we have $\l_1(\H)<\l_1(\H_1)$. Let $v_1,v_2,v_3,v_4,v_5$ be the core vertices of $\H_1$, where $v_3$ is the common vertex of the two cycles in $\H_1$. Now for the Perron eigen vector $X$ of $\H_1$ let $X_{v_t}=$max$\{X_{v_i}|i=1,2,3,4,5\}$. We have two cases, either $t=3$ or $t\neq 3$. If $t=3$, then we move all the pendant edges from $v_1,v_2,v_4,v_5$ to $v_3$. If $t\neq 3$, let $C_L$ be the cycle in $\H_1$ not containing the vertex $v_t$. Then we move all the pendant edges from $v_i(i\neq t)$ to $v_t$ and two edges of the cycle $C_L$ incident to the vertex $v_3$ from $v_3$ to $v_t$, respectively. In both cases we have $BC(k-6)$ as the resultant hypergraph and hence $\l_1(\H)<\l_1(BC(k-6))$.
		
		\item \textbf{Case 2: $\mathbf{C_{L_1}}$ and $\mathbf{C_{L_2}}$ have no common vertex.} Let $v_1$ and $v_2$ be two core vertices in $C_{L_1}$ and $C_{L_2},$ respectively. Now for the Perron eigen vector $X$ if $X_{v_1}\geq X_{v_2}$ ($X_{v_2}> X_{v_1}$) then we move all the edges, incident to $v_2,$ ($v_1$) from $v_2$ to $v_1$ (from $v_1$ to $v_2$). Let $\H_2$ be the resultant hypergraph. Then $\l_1(\H)<\l_1(\H_2)$. Now we apply the procedure of case 1 for the hypergraph $\H_2$ and this completes the proof.
	\end{itemize}
	\item Let $C_{L_1}=v_1e_1v_2e_2v_3e_3v_1, C_{L_2}=v_1e_4v_4e_5v_5e_6v_1$ be the two cyles of $BC(k-6)$ and let $f_1,f_2,\dots,f_{k-6}$ be the pendant edges of $BC(k-6)$ containing the vertex $v_1$. Also let $V_1=\{v_1\},V_2=\{v_2,v_3,v_4,v_5\},V_3=\cup_{1}^{k-6}f_i\backslash \{v_1\},V_4=e_1\cup e_3\cup e_4\cup e_6\backslash\{v_1,v_2,v_3,v_4,v_5\}$ and $V_5=e_2\cup e_5 \backslash \{v_2,v_3,v_4,v_5\}.$ Then the set $\{V_1,V_2,V_3,V_4,V_5\}$ forms an equitable partition for the hypergraph $BC(k-6)$. Now let $(\l_1,X)$ be the Perron eigenpair of $A_{BC(k-6)}$ and $\l=(m-1)\cdot\l_1(BC(k-6))$. Note that $v_2,v_3,v_4$ and $v_5$ are in the same part of the equitable partition. So we have $X_{v_2}=X_{v_3}=X_{v_4}=X_{v_5}.$. Now if $X_{v_1}\leq X_{v_2}$, then we move the edge $f_1$ from $v_1$ to $v_2$, and that gives us another bicyclic linear hypergraph with the spectral radius larger than the spectral radius of $BC(k-6)$. Hence we have $X_{v_1}> X_{v_2}$.\\
	Now for any vertex $v\in V_3\backslash \{{v_1}\}$ we have
	\begin{align*}
		\l X_v=(m-2)X_v+X_{v_1}\\
		\implies X_v=\dfrac{1}{\l-m+2}X_{v_1}.
	\end{align*}
	For any $v\in{V_4}$ we have 
	\begin{align*}
		\l X_v&=(m-3)X_v+X_{v_1}+X_{v_2}\\
		\implies X_v&=\dfrac{X_{v_1}+X_{v_2}}{\l-m+3}.
	\end{align*}	
	Using these we have 
	\begin{align}\label{Eq1}
		&\l X_{v_1}=(k-6)(m-1)X_v+\dfrac{4(m-2)}{\l-m+3}(X_{v_1}+X_{v_2})+4X_{v_2}\notag \\
		\implies & \left(\l-\dfrac{(k-6)(m-1)}{\l-m+2}-\dfrac{4(m-2)}{\l-m+3}\right)X_{v_1}=\dfrac{4(\l+1)}{\l-m+3}X_{v_2}
	\end{align}	
	And 
	\begin{align}\label{Eq2}
		&\l X_{v_2}=\dfrac{m-2}{\l-m+3}(X_{v_2}+X_{v_1})+\dfrac{m-2}{\l-m+3}(X_{v_2}+X_{v_3})+X_{v_1}+X_{v_3}\notag \\
		\implies &\left((\l-1)(\l-m+3)-3(m-2)\right)X_{v_2}=(\l+1)X_{v_1}.
	\end{align}
	From the equations \eqref{Eq1} and  \eqref{Eq2} we have 
	
	\begin{align*}
		&\left(\l-\dfrac{(k-6)(m-1)}{\l-m+2}-\dfrac{4(m-2)}{\l-m+3}\right)\left((\l-1)(\l-m+3)-3(m-2)\right)=\dfrac{4(\l+1)^2}{(\l-m+3)}\\
		\implies &\left(\l^2-(m-2)\l-2m+3\right)\left((\l-m+3)\left(\l(\l-m+2)-(k-2)(m-1)+4\right)+4(m-2)\right)\\ &\ \ \ \ \ \ \ \ \ -4(\l+1)^2(\l-m+2)=0.
	\end{align*}
	This completes the proof for the first part of $\bf(b)$. \\
	Now we know that $X_{v_1}>X_{v_2}$, so from the equation (\ref{Eq1}) we have
	\begin{align*}
		&\{\l-\dfrac{(k-6)(m-1)}{\l-m+2}-\dfrac{4(m-2)}{\l-m+3}\}<\dfrac{4(\l+1)}{\l-m+3}\\
		\implies& \l-\dfrac{(k-6)(m-1)}{\l-m+2}-\dfrac{4(m-2)}{\l-m+2}<\dfrac{4(\l+1)}{\l-m+3}<\dfrac{4(\l+1)}{\l-m+2}\\
		\implies& \l<\left(m+2+\sqrt{(m+2)^2+4(k-2)(m-1)}\right)/2.
	\end{align*}
	This completes the proof.
	\end{enumerate}
\end{proof}

Let $B_2C(l_1,l_2)$ be the hypergraph obtained by attaching $l_2$ pendent edges to a vertex of degree $2$ of $BC(l_1).$

\begin{lemma}\label{B_2C}
	\begin{enumerate}[$(i)$]
	    \item Let $l_1\geq l_2\geq 1$. Then $\l_1(B_2C(l_1,l_2))<\l_1(B_2C(l_1+1,l_2-1))$.
		\item Let $l_1>l_2\geq 0$. Then $\l_1(B_2C(l_2,l_1))<\l_1(B_2C(l_1,l_2)).$
	\end{enumerate}
\end{lemma}
\begin{proof}
	\begin{enumerate}[$(i)$]
	\item Let $u,v$ be the vertices of degree $l_1+4$ and $l_2+2$, respectively in $B_2C(l_1,l_2)$. Then for the Perron eigenvector $X$ of $B_2C(l_1,l_2),$ we have either $X_{u}\leq X_{v}$ or $X_{v}<X_{u}$. If 
	$X_{u}\leq X_{v},$ then we move a pendant edge from $u$ to $v$. Now let $X_{v}< X_{u}$. Let $e_1,e_2$ be two non-pendant edges containing the vertex $v$, but not in the cycle, containing the vertex $u$. Now we move $e_1,e_2$ and $l_1-l_2+1$ non-pendant edges from $v$ to $u$. In both cases we get $B_2C(l_1+1,l_2-1)$ as the resultant hypergraph. Therefore we have $\l_1(B_2C(l_1,l_2))<\l_1(B_2C(l_1+1,l_2-1))$.
	\item Similar to the first part.
	\end{enumerate}
\end{proof}

Now the following lemma shows that $B_2C(k-7,1)$ is the probable candidate which have the second largest spectral radius in $\mathcal{BC}$

	\begin{lemma} \label{BCL}
		Let $\H\in \PH\backslash \{B_2C(k-6,0),B_2C(0,k-6)\}$ be any bicyclic linear hypergraph with $k(\geq 8)$-edges. Then  $\l_1(\H)\leq \l_1(B_2C(k-7,1))$ and equality holds only when $\H=B_2C(k-7,1)$.
	\end{lemma}

\begin{proof}
Let $C_{L_1}$ and $C_{L_2}$ be the two cycles in $\H$. So we have two cases
\begin{enumerate}[$(a)$]
	\item \textbf{Case 1: $\mathbf{C_{L_1}}$ and $\mathbf{C_{L_2}}$ have a common vertex, say $\mathbf{v}$.} 
	\begin{itemize}

	 \item \textbf{Subcase I : Length of both cycles is 3.} Also let $C_{L_1}=ve_1v_1e_2v_2e_3v$ and  $C_{L_2}=ve_4v_3e_5v_4e_6v$. Now we have two cases either  $d(v)=4$ or $d(v)>4$. 
	 \begin{enumerate}[(i)]
	 	\item $\mathbf{d(v)=4}:$ \\
	 	We suppose exactly one of the vertices $v_i's$ (say $v_1$) is of degree greater than two. Since $\H\neq B_2C(k-6,0),B_2C(0,k-6)$ it follows that there is a non-pendant edge (not in the cycles) say $e$ containing $v_1$. Also let $u(\neq v_1)\in e$ be another vertex of degree greater than one. Let $\H_1$ be the hypergraph obtained from $\H$ by releasing all the non-pendant edges (not the edges of the cycles) except $e$. Then for the Perron eigenvector $X$ of the hypergraph $\H_1$ we have either $X_{v}\geq X_{u}$ or $X_{v}< X_{u}$. If $X_v\geq X_u$ then we move all the pendant edges from $u$ to $v$. If $X_{v}< X_{u}$ then move the edge $e_3$ from $v$ to $u$ and then release the edge $e_1$ otherwise move the edges (except e) containing $u$ from $u$ to $v$. In both cases we have $B_2C(l_1,l_2)$ as the resultant hypergraph.\\
	 	Now suppose more than one of the vertices $v_i's$ are of degree greater than two and let $d(v_1)\geq 3.$ Let $\H_1$ be the hypergraph obtained from $\H$ by releasing all non-pendant edges (not in the cycles) of $\H$. For the Perron eigenvector $X$ of $\H$ let $X_{v_t}=\text{max}\{X_{v_2},X_{v_3},X_{v_4}\}.$ Let $\H_2$ be the hypergraph obtained from $\H_1$ by moving all the edges (not the edges of cycles) of $\H_1$ from $v_2,v_3,v_4$ to $v_t$. Let $v_t\in V(C_{L_t})$ where either $t=1$ or $t=2$. Then for the Perron eigenvector $X'$ of $\H_2$ we have either $X'_v\geq X'_{v_t}$ or $X'_v< X'_{v_t}$. If $X'_v\geq X'_{v_t}$ then we move all the pendant edges from $v_t$ to $v$. On the other hand if $X'_v< X'_{v_t}$ then move the edges of $C_{L_t},$ not incident to $v_t,$ from $v$ to $v_t.$ In both the cases, we have $B_2C(l_1,l_2)$ as the resultant hypergraph. Then by Lemma \ref{B_2C} we have $\l_1(\H)<\l_1(B_2C(l_1,l_2))<\l_1(B_2C(k-7,1))$.\\
	 	
	 	\item $\mathbf{d(v)>4:}$ Now we have either $d(v_i)=2$ for all $i=1,2,3,4$ or at least one of $v_i's$ is of degree greater than two. \\
	 	First suppose that $d(v_i)=2$ for all $i=1,2,3,4$. Since $\H\neq B_2C(k-6,0),B_2C(0,k-6)$ it follows that there exists a non-pendant edge say $f$ containing $v.$ Now we relaese all the non-pendant edges of (except $f$ and the edges of the cycles) and let $u_1(\neq v)\in f$ with $d(u_1)>1.$  For the Perron eigenvector $X$ of the resultant hypergraph $\H_3$ we have either $X_{v_1}\geq X_{u_1}$ or $X_{v_1}< X_{u_1}$. If $X_{v_1}< X_{u_1}$ then move the edge $e_2$ from $v_1$ to $u_1$ otherwise move the edges (except $f$) containing $u_1$ from $u_1$ to $v_1$. In both cases we have $B_2C(l_1,l_2)$ as the resultant hypergraph for some $l_1,l_2$.\\
	 	Next suppose that at least one of $v_i's$ has degree greater than two. Now for the Perron eigenvector $X$ of $\H,$ let $X_{v_t}=\text{Max}\{X_{v_1},X_{v_2},X_{v_3},X_{v_4}\},$ and move all the edges (not in the cycles) from $v_1,v_2,v_3,v_4$ to $v_t$. Then release all the non-pendant edges (not in the cycles) and we get $B_2C(l_1,l_2)$ as the resultant hypergraph.
\end{enumerate}
 
 \item \textbf{Subcase II : Length of $\mathbf{C_{L_1}}$ is greater than three and length of $\mathbf{C_{L_2}}$ is 3.}\\
   Since $k\geq 8$ it follows that length of $C_{L_1}$ is greaer than four or there exists an edge other than the edges of the cycles. \\
 \textbf{(i) Length of $\mathbf{C_{L_1}}$ is 4:} \\
 Let $e$ be an edge other than the edges of the cycles. Then we have either $e$ is incident to $v$ or $e$ is incident to one of the other core vertices of the cycles. First suppose that $e$ is incident to $v$. Now we release the edges of the cycle $C_{L_1}$, not containing vertex $v$ untill the length of the cycle becomes $3$. Let $\H_4$ be the resultant hypergraph. Then $\H_6\neq B_2C(k-6,0),B_2C(0,k-6)$ and length of the cycles in $\H_6$ is 3. So by subcase I, we  get the result. Next we suppose that $e$ is incident to a core vertex (other than $v$) say $v_1.$ If $v_1$ is a vertex of $C_{L_2}$ then release the edges of the cycles untill the length of the cycle becomes $3$. On the otherhand if $v_1$ is a vertex of $C_{L_1}$, then we release the edges of $C_{L_1}$ not containing the vertex $v_1$ untill the length of cycle become $3$. Then by subcase I, we have the required result. \\
 
 \textbf{(ii) Length $\mathbf{C_{L_1}}$ is greater than 4:}\\
 	 Let $v_1,v_2$ be the core vertices of $C_{L_1}$ adjacent to $v$. Now release the edges of $C_{L_1}$ incident to $v_1$ and $v_2$ (but not containing the vertex $v$) at $v_1$ and $v_2$, respectively. So we have $d(v_1),d(v_2)>3$. Next we release all other edges of $C_{L_1}$ not containing the vertex $v$ so that $d(v_1),d(v_2)>3$ and length of the cycle become $3$. Let $\H_5$ be the resultant hypergraph. Then $\l_1(\H)<\l_1(\H_5)$ and $\H_5\neq B_2C(k-6,0),B_2C(0,k-6)$  so by subcase I we have the required result.\\

 	\item \textbf{Subcase III : Length of $\mathbf{C_{L_2}}$ greater than 3: }\\
 	 We release the edges of $C_{L_1}$ and $C_{L_2}$ not containing the vertex $v$ untill the length of the cycles become 3. Then by subcase 1, of case 1, we get the required result.
 \end{itemize}

\item \textbf{Case 2 : $\mathbf{C_{L_1}}$ and $\mathbf{C_{L_2}}$ has no common vertex.}\\
 At first we release the non-pendant edges of the cycles, so that the lentgh of the both cycles become $3.$ Let $C_{L_1}=v_1e_1v_2e_2v_3e_3v_1$ and $C_{L_2}=u_1f_1u_2f_2u_3f_3u_1$. Since $\H$ is bicyclic it follows that there exists an unique loose path $P_L$ connecting $C_{L_1}$ and $C_{L_2}$. Now we have two subcases 
    \begin{itemize}
    	\item \textbf{Subcase I:Length of $\mathbf{P_L}$ is one.}\\ Let $e$ be the edge containing $v_1,u_1$. Since $k\geq 8$, there exists an pendant edge say $e'$ incinent to one of the core vertices. \\
    	Suppose $e'$ is incident to $v_1$, then for the Perron eigenvector $X$ of $\H$ we have either $X_{v_1}\geq X_{u_2}$ or $X_{v_1}< X_{u_2}$. If $X_{v_1}< X_{u_2}$ then we move the edges $e',e_1,e_2$ from $v_1$ to $u_2$ otherwise we move the edge $f_2$ from $u_2$ to $v_1$. Let $\H_6$ be the resultant hypergraph. Then $\l_1(\H)<\l_1(\H_6)$ and by subcase 1 we have the required result.\\
    	If $e'$ is incident to any one of $v_2,v_3,u_2,u_3$, then let $\H_7$ be the hypergraph obtained from $\H$ by releasing the edge $e$. Then $\l_1(\H)< \l_1(\H_7)$ and $\H_7\notin\{B_2C(k-6,0),B_2C(0,k-6)\}$. Then by subcase 1 we get the required result. 
    	\item \textbf{Subcase II: Length of $\mathbf{P_L}$ is greater than one.} Let $P_L$ contains $v_1$ and $u_1.$ Then for the Perron eigenvector $X$ of $\H$ we have either $X_{v_1}\geq X_{u_1}$ or  $X_{v_1}< X_{u_1}$. If 
    	 $X_{v_1}\geq X_{u_1}$ then we move the edges $f_1,f_3$ from $u_1$ to $v_1$ otherwise we move the edges $e_1,e_3$ from $v_1$ to $u_1$. Then the result follows by subcase 1. 
    \end{itemize}
	
 \end{enumerate}
\end{proof}
Now we show $B_2C(k-7,1)$ is the hypergraph, having the second largest spectral radius in $\mathcal{BC}.$ 
\begin{theorem} Let $\H\in \mathcal{BC}$ and $\H\neq B_2C(k-6,0)$ be any bicyclic linear hypergraph with $k(\geq 8)$-edges. Also let $(k-12)(m-1)\geq 56$. Then  $\l_1(\H)\leq \l_1(B_2C(k-7,1))$ and the equality holds only when $\H=B_2C(k-7,1)$.
\end{theorem}
\begin{proof}
	Let $C_{L_1}=ve_1v_1e_2v_2e_3v$ and  $C_{L_2}=ve_4v_3e_5v_4e_6v$ be the cycles in $B_2C(0,k-6).$
	Let $u$ be a pendant vertex in $B_2C(0,k-6)$ adjacent to $v_2$. Also let $X$ be the Perron eigenvector of $B_2C(0,k-6)$ and $\l_1(B_2C(0,k-6))(m-1)=\l.$ Note that $X_{v_3}=X_{v_4}<X_{v_2}<X_{v}.$
	Now from the eigenvalue equation we have 
	\begin{align*}
		\left(\left(\l-\dfrac{\l+2m-3}{\l-m+3}\right)\left(\l-\dfrac{4(m-2)}{\l-m+3}\right)-\dfrac{2(\l+1)^2}{\l-m+3}\right) X_{v_3}<(\l-m+2)\dfrac{2(\l+1)^2}{\l-m+3}X_{u}
	\end{align*}
Again for $(k-12)(m-1)\geq 56,$ we have 
$\left(\left(\l-\dfrac{\l+2m-3}{\l-m+3}\right)\left(\l-\dfrac{4(m-2)}{\l-m+3}\right)-\dfrac{2(\l+1)^2}{\l-m+3}\right)>(\l-m+2)\dfrac{2(\l+1)^2}{\l-m+3}$ that is, $X_{v_3}<X_u.$ Let $\H_1$ be the hypergraph obtained from $B_2C(0,k-6)$ by moving the edge $e_5$ from $v_3$ to $u$. For the Perron eigenvector $Y$ of $\H_1$ we have $Y_{v}<Y_{v_1}$. Let $\H_2$ be the hypergraph obtained from $\H_1$ by moving the edge $e_6$ from $v$ to $v_1.$ Then $\l_1(\H_1)<\l_1(\H_2)$ and $\H_2=B_2C(k-7,1)$.
	\end{proof}

\section{Tricyclic linear hypergraphs with largest and second-largest spectral radii in $\PH$}

	We know that an $m$-uniform hypergraph $\H(V,E)$ with $n$-vertices and $k$-edges is
\begin{itemize}
	\item Unicyclic if and only if $n=k(m-1)$.
	\item If $\H$ is $2$-cyclic, then $n=k(m-1)-r+1$ but the converse is not true, i.e., there exists linear hypergraph $\H_1$ satisfying $n=k(m-1)-2+1$ and $\H_1$ is not $2$-cyclic.
	For example consider the hypergraph $\H(V,E)$ with $V=\{1,2,\dots,14\}$ and $E=\{\{1,2,3,4\},\{4,5,6,7\},\{7,8,9,10\},\{10,11,12,1\},\{7,13,14,1\}\}.$
\end{itemize}

Note that edge releasing, moving operations are vertex set preserving.

\begin{proposition}
	Let $\H(V,E)\in \PH$ be any $3$-cyclic linear hypergraphs with $n$-vertices and $k$-edges. Then either $n=k(m-1)-1$ or $n=k(m-1)-2.$
\end{proposition}
\begin{proof}
	Let $C_{L_1},C_{L_2},C_{L_3}$ be the three cycles in $\H$ and length of $C_{L_i}$ be $l_i$. Let $\H_1$ be the hypergraph obtained by releasing all the non-pendant edges (but not the edges in the cycles). Then $|V(\H)|=|V(\H_1)|.$ We have following cases
	\begin{itemize}
		\item \textbf{Case 1 : For all $\mathbf{i(\neq) j\in\{1,2,3\},C_{L_i}}$ and $\mathbf{C_{L_j}}$ has no common edge.}\\
		 Now we have subcases
		\begin{enumerate}[(i)]
			
			\item \textbf{Subcase I : $\mathbf{C_{L_1},C_{L_2},C_{L_3}}$ has a common vertex say $\mathbf{v}$.}
			\begin{align*}
				|V|&=\{(|V(C_{L_1})|+|V(C_{L_2})|+|V(C_{L_3})|)-2\}\\
				&\ \ \ \ +(k-l_1-l_2-l_3)(|e|-1)\\
				&=k(m-1)-2.
			\end{align*}	
			
			\item \textbf{Subcase II : $\mathbf{C_{L_1},C_{L_2},C_{L_3}}$ has no common vertex.} 
			We move the edges of the cycles so that, they share a common vertex. Then by using subcase I
			we have $|V|=k(m-1)-2.$
		\end{enumerate}
		
		\item \textbf{Case 2: \textbf{ $\mathbf{C_{L_1}}$ and $\mathbf{C_{L_2}}$} has common edges.}\\ Let  $e_1,e_2,\dots,e_p$ be the  common edges of $C_{L_1}$ and $C_{L_2}$. So we can write\\
		$C_{L_1}=v_1e_1v_2e_2v_3\dots v_pe_pu_1f_1u_2f_2\dots u_qf_qv_1$ and $C_{L_2}= v_1e_1v_2e_2v_3\dots v_pe_pu_1g_1w_1g_2\dots w_{r-1}g_rv_1$ for some $q,r$. Here $k_1=p+q$ and $k_2=p+r$. Since $\H$ is $3$-cyclic we have $C_{L_3}=v_1f_qu_qf_{q-1}u_{q-1}f_{q-2}\dots f_1u_1g_1w_1g_2\dots w_{r-1}g_rv_1$. Hence $\H$ has $l=k-\{(k_1+k_2-p)+(k_1-p+k_2-p)\}=k-(k_1+k_2+k_3-p)$ pendant edges.
		So we have
		\begin{align*}
			|V|&=\text{(Number of vertices in $C_{L_3}$)}+\text{(Number of vertices in $\cup_{i=1}^pe_i \backslash \{v_1,u_1\}$)}\\
			& \ \ \ \ +l(m-1)\\
			&=k(m-1)-1.
		\end{align*} 
		
	\end{itemize}
\end{proof}

Let $\H$ be an $m$-uniform $3$-cyclic linear hypergraph with $n$-vertices and $k$-edges. We say $\H$ is of
\begin{itemize}
	\item Type I if $n=k(m-1)-1$.
	\item Type II if $n=k(m-1)-2$.
\end{itemize}
Let ${\R}_1$ and ${\R}_2$ be the collection of $3$-cyclic Type I linear hypergraphs and Type II linear hypergraphs in $\PH$, respectively. 
\begin{note}
	Let $\H\in \PH$ be a three cyclic linear $m$-uniform hypergraph. Then $\H$ is of Type I if and only if any two cycles in $\H$ has a common edge. 
\end{note}

\subsection{Tricyclic type I linear hypergraphs with largest and second-largest spectral radii }

Let $T_1C(l_1,l_2,l_3,l_4))\in{\R}_1$ be the hypergraph consisting three cycles $v_1e_1v_2e_2v_3e_3v_1$,~$v_1e_3v_3e_4v_4e_5v_1$,~\\ $v_1e_1v_2e_2v_3e_4v_4e_5v_1$ and $l_1,l_2,l_3,l_4$ pendant edges at $v_1,v_2,v_3,v_4$, respectively.

\begin{lemma}\label{TC1}
Let $l_1\geq 1.$ Then $\l_1(T_1C(0,l_1,0,0))<\l_1(T_1C(l_1,0,0,0)).$
\end{lemma}

\begin{proof}
	Let $\H=T_1C(0,l_1,0,0)$. Now for the Perron eigenvector $X$ of $\H$ we have either $X_{v_2}>X_{v_1}$ or $X_{v_2}\leq X_{v_1}$. If $X_{v_2}>X_{v_1}$ then we move the edge $e_5$ from $v_1$ to $v_2$, otherwise we move all the pendant edges from $v_2$ to $v_1$. 
\end{proof}

\begin{theorem}
	Let $\H\in {\R}_1$ with $k$ edges. Then
	\begin{enumerate}[(i)] 
		\item $\l_1(\H)\leq \l_1(T_1C(k-5,0,0,0))$ and the equality holds only when $\H=T_1C(k-5,0,0,0).$ \item $\l_1(T_1(k-5,0,0,0))<\{m+1+\sqrt{(m+1)^2+4(k-3)(m-1)}\}/(2m-2).$
	\end{enumerate}	
\end{theorem}
\begin{proof}
Let $C_{L_1},C_{L_2},C_{L_3}$ be the cycles in $\H$ of lengths $l_1,l_2,l_3,$ respectively and  $l_1\leq l_2\leq l_3$. Let $e_1,e_2,\dots,e_p$ be the common edges of $C_{L_1}$ and $C_{L_2}.$ Since $\H$ is $3$-cyclic we have $E(C_{L_3})=E(C_{L_1})\cup E(C_{L_2})\backslash \{e_1,e_2,\dots,e_p\}$ and $l_3=l_1+l_2-2p.$ Note that $2p\leq l_1$ and $2p\leq l_2.$ Release the edges $e_2,\dots,e_p$ one by one. Let $\H_1$ be the resultant hypergraph and $C_{L_1},C_{L_2},C_{L_3}$ becomes $C'_{L_1},C'_{L_2},C'_{L_3},$ respectively in $\H_1$. Then we have length of $C'_{L_1},$ $l_1'=l_1-(p-1)\geq 2p-(p-1)=p+1.$ Also length of $C'_{L_2},l_2'\geq p+1$. So in $\H_1$ the cycles $C'_{L_1},C'_{L_2}$ has only one common edge say $e.$ Let $H_2$ be the hypergraph obtained by releasing some edges of the cycles $C'_{L_1},C'_{L_2}$ so that the length of the cycles becomes $3.$
Let $\H_4$ be the hypergraph obtained from $\H_3$ by releasing all the non-pendant edges at the core vertices of the cycles. Now for the Perron eigenvector $X$ of $\H_4$ let $X_{v_t}=\text{max}\{X_{v_i}|i=1,2,3,4\}.$ Let $\H_5$ be the hypergraph obtained from $\H_4$ by moving the pendant edges from $v_i's$ to $v_t$. We have $\H_5=T_1C(k-5,0,0,0)$ or $\H_5=T_1C(0,k-5,0,0)$ and so using lemma \ref{TC1} we have $\l_1(\H)<\l_1(\H_1)<\l_1(\H_2)<\l_1(\H_3)<\l_1(\H_4)<\l_1(\H_5)<\l_1(T_1C(k-5,0,0,0)).$
For the next part let $f_1,f_2,\dots,f_{k-5}$ be the pendant edges of $T_1(k-5,0,0,0)$. Let $V_1=\{v_1\},V_2=\{v_2,v_4\},V_3=\{v_3\},V_4=e_1\cup e_5\backslash \{v_1,v_2,v_4\},V_5=e_2\cup e_4\backslash \{v_2,v_3,v_4\},V_6=e_3\backslash \{v_1,v_3\}$ and $V_7=\cup_{i=1}^{k-5}f_i\backslash\{v_1\}.$ Then $P=\{V_1,V_2,\dots V_7\}$ forms an equitable partition.
Let $(X,\l_1)$ be the Perron eigen-pair  of $A_{T_1C(k-5,0,0,0)}$ and $\l=(m-1)\cdot\l_1$. Then we have
\begin{align*}
 \left(\l-\dfrac{(k-5)(m-1)}{\l-m+2}-\dfrac{3(m-2)}{\l-m+3}\right)X_{v_1}=\dfrac{\l+1}{\l-m+3}(2X_{v_2}+X_{v_3})
\end{align*}
Now if $X_{v_3}\leq X_{v_2}$ then we move the edge $e_4$ from $v_3$ to $v_2$ and let $\H_2$ be the resultant hypergraph. Then $\l_1(T_1(k-5,0,0,0))<\l_1(\H_2)$. But $H_2=T_1(k-5,0,0,0)$. Thus we have $X_{v_2}<X_{v_3}$ and so from the above equation we have  
\begin{align*}
	& \left(\l-\dfrac{(k-5)(m-1)}{\l-m+2}-\dfrac{3(m-2)}{\l-m+3}\right)X_{v_1}<\dfrac{3(\l+1)}{\l-m+3}X_{v_3}
\end{align*}	
 If $X_{v_1}\leq X_{v_3}$ then we move all the pendant edges from $v_1$ to $v_3$ and get $\H_3=T_1C(k-5,0,0,0)$ as the resultant hypergraph with $\l_1(T_1C(k-5,0,0,0))<\l_1(\H_3)$ which is not possible. Thus we have $X_{v_1}>X_{v_3}$ and so 
\begin{align*}
	&\l-\dfrac{(k-5)(m-1)}{\l-m+2}-\dfrac{3(m-2)}{\l-m+3}<\dfrac{3(\l+1)}{\l-m+3}\\
	\implies& \l<\dfrac{m+1+\sqrt{(m+1)^2+4(k-3)(m-1)}}{2}.
\end{align*}
and hence $$\l_1<\dfrac{m+1+\sqrt{(m+1)^2+4(k-3)(m-1)}}{2(m-1)}$$. This completes the proof.
\end{proof}
\begin{lemma}
		Let $\H\in \PH$ be a tricyclic Type I hypergraph with $k$-edges. Also let $\H\neq T_1C(k-5,0,0,0)$. Then $\l_1(\H)\leq \text{max}\{ \l_1(T_1C(k-6,0,1,0)),\l_1(T_1C(0,k-6,0,0))\}.$ 
\end{lemma}
\begin{proof}
Let $C_{L_1},C_{L_2},C_{L_3}$ be the cycles in $\H$ of lengths $l_1,l_2,l_3,$ respectively and  $l_1\leq l_2\leq l_3$. Let $f_1,f_2,\dots,f_p$ be the common edges of $C_{L_1}$ and $C_{L_2}.$ Since $\H$ is $3$-cyclic of type I we have $E(C_{L_3})=E(C_{L_1})\cup E(C_{L_2})\backslash \{f_1,f_2,\dots,f_p\}$ and $l_3=l_1+l_2-2p.$ Note that $2p\leq l_1$ and $2p\leq l_2.$
\begin{enumerate}[$(i)$]
	\item \textbf{Case I : $\mathbf{l_1=3.}$} \\
	Then $p=1.$ Let $C_{L_1}=v_1e_1v_2e_2v_3f_1v_1$. We have two subcases,\\
	\textbf{Subcase I : $\mathbf{l_2=3:}$}
	 Let $C_{L_2}=v_1f_1v_3e_3v_4e_4v_1$. We have following cases
	
	\begin{itemize}
		\item $\mathbf{d(v_1)=d(v_3)=3:}$\\
		 For the Perron eigenvector $X$ of $\H$ we have either $X_{v_2}>X_{v_4}$ or $X_{v_2}\leq X_{v_4}.$  Let $\H_1$ be the hypergraph obtained by moving the edges (not in the cycles) from $v_4$ to $v_2$ (if $X_{v_2}>X_{v_4}$) or from $v_2$ to $v_4$ (if $X_{v_2}\leq X_{v_4}$). Then let $\H_2$ be the hypergraph obtained from $\H_1$ by releasing all the non-pendant edges (not in the cycles). Then $\H_2=T_1C(0,k-6,0,0)$ and we have
		$\l_1(\H)<\l_1(\H_2)=\l_1(T_1C(0,k-6,0,0)).$
		\item \textbf{$\mathbf{d(v_1)>3}$ or $\mathbf{d(v_3)>3:}$} Let $d(v_1)>3$. \\
		If $d(v_3)$ is also greater than $3$, then let $\H_1$ be the hypergraph obtained from $\H$ by releasing the non-pendant edges (not in the cycles). Now for 
		  the Perron eigenvector $X$ of $\H$ let $X_t=$max$\{X_{v_2},X_{v_3},X_{v_4}\}.$ Then let $\H_2$ be the hypergraph obtained from $\H_1$ by moving the pendant edges from $v_2$ and $v_4$, to $v_3$ (if $t=3$) or from $v_3$ and $v_4,$ to $v_2$ (if $t=2$). Thus we have $\H_2=T_1C(c_1,0,c_3,0)$ or $\H_2=T_1C(c_1,c_2,0,0)$ for some $c_1,c_2\geq 1$ and hence  $\l_1(\H)\leq \l_1(\H_2)\leq \l_1(T_1C(k-6,0,1,0)).$\\
		   Now if $d(v_3)=3$ then we have two possibilities\\
		  (a) $\mathbf{d(v_2)=d(v_4)=2:}$ Then there exists a non-pendant edge (not in the cycles) say $f'$ containing $v_1.$ Let $u(\neq v_1)\in f'$ be a non-pendant vertex. Then for the Perron-eigenvector $X$ of $\H$ we have either $X_{v_3}\geq X_u$ or  $X_{v_3}< X_u$. Let $\H_3$ be the hypergraph obtained from $\H$ by moving the edges (not $f'$) incinent to $u$ (if $X_{v_3}\geq X_u$) or the edges $e_2 $ and $e_3$ from $v_3$ to $u$ (if $X_{v_1}< X_u$ ). Let $\H_4$ be the hypergraph obtained from $\H_3$ by releasing the non-pendant (not in the cycles) edges. Then  $\l_1(\H)<\l_1(\H_4)<\l_1(T_1C(k-6,0,1,0)).$  \\
		  (b) $\mathbf{d(v_2)>2}$ \textbf{or} $\mathbf{d(v_4)>2:}$ Let $\H_5$ be the hypergraph obtained from $\H$ by releasing the non-pendant (not in the cycles) edges. Let $\H_6$ be the hypergraph obtained from $\H_5$ by moving the pendant edges from $v_4$ to  (if $X_{v_2}>X_{v_4}$) or from $v_2$ to $v_4$ (if $X_{v_2}\leq X_{v_4}$). Then $\H_6=T_1C(c_1,c_2,0,0)$ for some $c_1,c_2\geq 1.$ Thus we have  $\l_1(\H)<\l_1(T_1C(k-6,0,1,0)).$
	\end{itemize}
\textbf{Subcase II :: $\mathbf{l_2>3}$:} Let $\H_3$ be the hypergraph obtained from $\H$ by releasing the edges of $C_{L_2}$ ( not containing the vertices $v_1$ and $v_3$). Then $\l_1(\H)<\l_1(\H_4)$ and by subcase I we get the required result.
\item \textbf{Case II :: $\mathbf{l_1>3}$: } Let $C_{L_1}=v_1e_1v_2e_2\dots v_{l-1}e_lv_lf_1\dots f_pv_1.$ Also let $\H_5$ be the hypergraph obtained from $\H$ by releasing the edges of $C_{L_1}$ (not $f_1,f_2,\dots,f_p$) one time releasing at $v_1$ and other time at $v_3$ so that the length of the cycle become $p+1.$ Thus in $\H_5$ we have two cycles havinng one edge (say $f$) in common and let $l_1',(\leq)l_2'$ be the length of the cycles. If $l_1'=3$ then the result follows from case I. If $l_1'\neq 3$ then let $\H_6$ be the hypergraph obtained from $\H_5$ by releasing the edges of the smaller length cycle (but not the edge $e$) so that the length of the cycle become $3$. Then $\l_1(\H)<\l_1(\H_6)$ and then by subcase I we have the required result.
\end{enumerate}	
\end{proof}

\begin{theorem}
	Let $\H\in \PH$ be a tricyclic Type I hypergraph with $k$-edges. Also let $\H\neq T_1C(k-5,0,0,0)$. Then $\l_1(\H)\leq \l_1(T_1C(k-6,0,1,0))$ and the equality holds only when $\H=T_1C(k-6,0,1,0).$
\end{theorem}
\begin{proof}
Let $e$ be a pendant edge in $T_1C(0,k-6,0,0)$ and $u\in e$ adjacent to $v_2.$	Let $\l=(m-1)\l_1(T_1C(0,k-6,0,0))$. Then for the Perron eigenvector $X$ of $T_1C(0,k-6,0,0)$ we have	
\begin{align*}
	\left(\left(\l-\dfrac{\l+3m-5}{\l-m+3}\right)\left(\l-\dfrac{2(m-2)}{\l-m+3}\right)-\dfrac{2(\l+1)^2}{\l-m+3}\right)X_{v_4}=(\l-m+2)\dfrac{2(\l+1)^2}{\l-m+3}X_{u}.
\end{align*}
If $(k-11)(m-1)\geq 42$ then $\l\geq m+5$ and then
\begin{align*}
	\left(\left(\l-\dfrac{\l+3m-5}{\l-m+3}\right)\left(\l-\dfrac{2(m-2)}{\l-m+3}\right)-\dfrac{2(\l+1)^2}{\l-m+3}\right)>(\l-m+2)\dfrac{2(\l+1)^2}{\l-m+3}.
\end{align*} Thus we have $X_{v_4}<X_u.$ Let $\H_1$ be the hypergraph be obtained from $T_1C(0,k-6,0,0)$ by moving the edge $e_5$ from $v_4$ to $u.$ Then $\H_1=T_1C(k-7,1,0,0)$ and $\l_1(T_1C(0,k-6,0,0))<\l_1(T_1C(k-7,1,0,0))$. Again by the above lemma we have $\l_1(T_1C(k-7,1,0,0))<\l_1(T_1C(k-6,0,1,0)).$ This completes the proof.
\end{proof}

\begin{remark}\label{BT}
	\begin{enumerate}[$(i)$]
	\item $\l_1(BC(k-6))<\l_1(T_1C(k-5;0;0;0))$ for $k\geq 6$.
	\item 	Max$\{\l_1(B_2C(0,k-6)),\l_1(B_2C(k-7,1))\}<\l_1(T_1C(k-6,0,1,0))$.
	\end{enumerate}
	
\end{remark}

\subsection{Tricyclic Type II linear hypergraphs with largest and second-largest spectral radii }

Let $T_2C(c_1=l_1,c_2=l_2,\dots,c_7=l_7)\in {\R}_2$ be denote the hypergraph consisting three cycles $v_1e_1v_2e_2v_3e_3v_1$,~$v_1e_4v_4e_5v_5e_6v_1$, and $v_1e_7v_7e_8v_8e_9v_1$ and $l_i's$ be the numbers of pendant edges at $v_i$, respectively, for $i=1,2,\dots,7.$ Now we have the following theorem
\begin{theorem}\label{T_2C}
	Let $\H$ be a three cyclic type II linear $m$-uniform hypergraph with $k$ edges. Then 
	\begin{enumerate}[(i)]
		\item $\l_1(\H)\leq \l_1(T_2C(c_1=k-9))$ and the equality holds only when $\H=T_2C(k-9).$
		\item  $\l_1(\H)=\gamma/(m-1)$ and $\gamma<\left(m+4+\sqrt{(m+4)^2+4(k-3)(m-1)}\right)/2,$
		 where  $\gamma$ is the root of the polynomial $\left(\l^2-(m-2)\l-2m+3\right)\left((\l-m+3)\{\l(\l-m+2)-(k-3)(m-1)+6\}+6(m-2)\right)\\
		 -6(\l+1)^2(\l-m+2)$ with the largest modulus.
	\end{enumerate}
\end{theorem}
\begin{proof}
	\begin{enumerate}[$(i)$] 
	\item Similar to the proof of first part of the theorem \ref{Bicyclic}.
	\item
	Let $f_1,f_2,\dots,f_{k-9}$ be the pendant edges of $T_2C(k-9)$. Let $V_1=\{v_1\},V_2=\{v_2,v_2,\dots,v_7\},V_3=e_1\cup e_3\cup e_4 \cup e_6 \cup e_7 \cup e_9 \backslash \{v_1,v_2,\dots,v_7\},V_4=e_2\cup e_5 \cup e_7 \backslash \{v_2,v_3,\dots,v_7\}$ and $V_5=\cup_{i=1}^{k-9}f_i\backslash \{v_1\}$. Then the partition $P=\{V_1,V_2,\dots,V_5\}$ is equitable for the hypergraph $T_2C(c_1=k-9)$. For the Perron eigen-pair $(X,\l_1)$ of $T_2C(c_1=k-9)$ we have $X_{u_1}=X_{u_2}$ for any two vertices $u_1,u_2$ in the same part of the partition.
	
		 Let $\l=(m-1)\cdot\l_1$. Then for any $u\in V_5$ we have 
	\begin{align*}
		(\l-m+2)X_u=X_{v_1}
	\end{align*}
and for any $v\in V_3$
\begin{align*}
	X_v=(X_{v_1}+X_{v_2})/(\l-m+3).
\end{align*} 
Thus 
\begin{align}\label{T_2C1}
\left(\l-\dfrac{(k-9)(m-1)}{\l-m+2}-\dfrac{6(m-2)}{\l-m+3}\right)X_{v_1}=\dfrac{6(\l+1)}{\l-m+3}X_{v_2}.
\end{align}
Again, for any $w\in V_4$ we have 
\begin{align*}
	\l X_w&=(m-3)X_w+X_{v_2}+X_{v_2}\\
	\implies X_w&=2X_{v_2}/(\l-m+3).
\end{align*} 
So 
\begin{align}\label{T_2C2}
	\l X_{v_2}&= \dfrac{m-2}{\l-m+3}(X_{v_2}+X_{v_1}+X_{v_2}+X_{v_2})+X_{v_1}+X_{v_2}\notag \\
	\implies \left(\l^2-(m-2)\l-2m+3\right)X_{v_2}&=(\l+1)X_{v_1}
\end{align}
From $(\ref{T_2C1})$ and $(\ref{T_2C2})$ we have 
\begin{align*}
\left(\l^2-(m-2)\l-2m+3\right)\left((\l-m+3)\left(\l(\l-m+2)-(k-3)(m-1)+6\right)+6(m-2)\right)\\
-6(\l+1)^2(\l-m+2)=0.
\end{align*}
 For the second part, we have $X_{v_2}<X_{v_1}$ if not then move an edge $f_1$ from $v_1$ to $v_2$. Let $\H_2$ be the resultant hypergraph. Then $\H_2\in{\R}_2$ and $\l_1(T_2C(c_1=k-9))<\l(\H_2)$-contradicts first part of this theorem. Thus we have $X_{v_2}<X_{v_1}.$ Now from $(\ref{T_2C1})$ we get
 \begin{align*}
 	&\l-\dfrac{(k-9)(m-1)}{\l-m+2}-\dfrac{6(m-2)}{\l-m+3}<\dfrac{6(\l+1)}{\l-m+3}\\
 	\implies& \l<\left(m+4+\sqrt{(m+4)^2+4(k-3)(m-1)}\right)/2.
 \end{align*}
 \end{enumerate}
Thus the result follows.
\end{proof}
\begin{lemma}
		Let $\H$ be a three cyclic type II linear $m$-uniform hypergraph with $k$ edges and $\H\neq T_2C(c_1=k-9) $. Then $\l_1(\H)\leq\text{max}\{ \l_1(T_2C(c_1=k-10,c_2=1)),\l_1(T_2C(c_2=k-9))\}$.
\end{lemma}
\begin{proof}
	Similar to the proof of lemma \ref{BCL}.
\end{proof}

\begin{theorem}
	Let $\H$ be a three cyclic type II linear $m$-uniform hypergraph with $k$ edges and $(k-13)(m-1)\geq 46$. If $\H\neq T_2C(c_1=k-9) $ then $\l_1(\H)\leq \l_1(T_2C(c_1=k-10,c_2=1))$ and the equality holds only when $\H=T_2C(c_1=k-10,c_2=1).$
\end{theorem}

\begin{proof}
	Let $X$ be the Perron eigenvector of $T_2C(c_2=k-9)$. If $X_{v_1}\geq X_{v_2}$ then we have $\l_1(T_2C(c_1=k-10,c_2=1))>\l_1(T_2C(c_2=k-9))$. Now let $X_{v_1}<X_{v_2}$ and $\H_1$ be the hypergraph obtained from $T_2C(c_2=k-9)$ by moving the edges $e_7,e_9$ from $v_1$ to $v_2$. Suppose $e$ is a pendant edge of $\H_1$ and $u\in e$ is adjacent to $v_2.$ Now for the Perron eigenvector $Y$ of $\H_1$ we have 
		\begin{align*}
		\left(\left(\l-\dfrac{\l+2m-3}{\l-m+3}\right)\left(\l-\dfrac{4(m-2)}{\l-m+3}\right)-\dfrac{2(\l+1)^2}{\l-m+3}\right)Y_{v_4}<(\l-m+2)\dfrac{2(\l+1)^2}{\l-m+3}Y_{u}
	\end{align*}
Thus when $(k-13)(m-1)\geq 46$, we have $\l>m+6$ and hence 	
		\begin{align*}
		\left(\left(\l-\dfrac{\l+2m-3}{\l-m+3}\right)\left(\l-\dfrac{4(m-2)}{\l-m+3}\right)-\dfrac{2(\l+1)^2}{\l-m+3}\right)>(\l-m+2)\dfrac{2(\l+1)^2}{\l-m+3}.
	\end{align*}
	Therefore $Y_{v_4}<Y_u$. Let $\H_2$ be the hypergraph obtained from $\H_1$ by moving the edge $e_5$ from $v_4$ to $u.$ Then $\l_1(\H_1)<\l_1(\H_2)$. Let $Z$ be the Perron eigenvector of $\H_2.$ If $Z_{v_1}\geq Z_{v_2}$ then move the edge $e_4$ from $v_1$ to $v_2$ otherwise move the edges $e,e_7,e_9$ and all pendant edges (except one) from $v_2$ to $v_1.$ In both the cases we get $T_2C(c_1=k-10,c_2=1)$ as the resultant hypergraph.
\end{proof}
\section{Acknowledgements}
Amitesh is thankful to Arpita, Samiron, and Rajiv for fruitful discussions.

\end{document}